\documentclass[onefignum,onetabnum]{siamart220329}
\newsiamremark{remark}{Remark}
\usepackage[utf8]{inputenc}
\usepackage{soul}
\usepackage{comment}
\usepackage{amsmath, amssymb, amscd}
\usepackage{graphicx}
\usepackage{xcolor}
\usepackage{multirow}
\usepackage{subcaption}
\usepackage{bm}
\usepackage{algorithm}
\usepackage{algpseudocode}
\usepackage{multirow}
\usepackage{mathtools}
\usepackage{comment}
\usepackage{xspace}

\newcommand{\bb}[1]{\mathbb{#1}}
\newcommand{\mbf}[1]{\mathbf{#1}}

\newcommand{\mc}[1]{\mathcal{#1}}
\newcommand*{\bmat}[1]{\begin{bmatrix}#1\end{bmatrix}}
\newcommand*{\sbmat}[1]{\bigl[ \begin{smallmatrix}#1\end{smallmatrix} \bigr ]}
\newcommand{\orth}{\textup{\texttt{orth}}}
\newcommand{\range}{\textup{\texttt{range}}}
\newcommand{\MATLAB}{\textsc{Matlab}\xspace}

\DeclareMathOperator{\rank}{rank}

\DeclareMathOperator*{\argmin}{argmin}

\usepackage{todonotes}

\newlength{\subcolumnwidth}

\newcommand{\nextsubcolumn}[1][]{%
  \cr\noalign{\hfill}
  \if\relax\detokenize{#1}\relax\else\hsize=#1\setlength{\subcolumnwidth}{\hsize}\fi
}

\newbox{\bigpicturebox}

\title{Frequency-Based Reduced  Models from Purely Time-Domain Data via Data Informativity}
\author{Michael S. Ackermann\thanks{Department of Mathematics, Virginia Tech, Blacksburg , VA, 24061  (\email{amike98@vt.edu}).}
\and Serkan Gugercin\thanks{Department of Mathematics and Division of Computational Modeling and
    Data Analytics, Academy of Data Science, Virginia Tech, Blacksburg, VA 24061} (\email{gugercin@vt.edu}).}

\date{March 2023}

\begin{document}

\maketitle

\begin{abstract}
    {
    Frequency-based methods have been successfully employed in creating high-fidelity data-driven reduced order models (DDROMs) for linear dynamical systems. These methods require access to values (and sometimes derivatives) of the frequency-response function (transfer function) in the complex plane.  These frequency domain values can at times be costly or difficult to obtain (especially if the method of choice requires resampling); instead one may have  access to only time-domain input-output data. 
    The data informativity approach to moment matching  provides a powerful new framework for recovering the required frequency data from a single time-domain trajectory.
    In this work, we analyze and extend upon this framework, resulting in significantly improved conditioning of the associated linear systems, an error indicator, and the removal of an assumption that the system order is known.  This analysis leads to a robust algorithm for recovering frequency information from time-domain data, suitable  for large-scale systems.  We demonstrate the effectiveness of our algorithm by forming frequency-based DDROMs from time-domain data of several dynamical systems.}
\end{abstract}



\section{Introduction}
\label{sec:Introduction}
High-fidelity dynamical system models often have large state space dimensions, which leads to prohibitively long computation times and large storage requirements.  Model reduction is concerned with finding systems with much smaller state dimension that can well approximate the full-order model to vastly speed up computation and lower memory requirements.
Classical model reduction techniques involve projecting the full-order system dynamics onto a smaller subspace \cite{Antoulas2005ApproxDynamSys,AntoulasBG2020Book,BennerGW2015SurveyProjMethods,RozzaHP2008ReducedBasis}. Examples of such methods include balanced truncation \cite{Moore1981PrincipleComponentAnalysis,MullisR1976BalancedTruncation}, and projection based interpolatory techniques \cite{Antoulas2005ApproxDynamSys,AntoulasBG2020Book}, the reduced basis methods \cite{DrohmannHO2012ReducedBasis,GreplMNP2007ReducedBasis, HaasdonkO2008ReducedBasis, RozzaHP2008ReducedBasis}, and proper orthogonal decomposition \cite{BerkoozHL1993PODTurbulent,Sirovich1987POD}.  These intrusive methods have had great success in finding reduced order models, but are limited in applicability as they require knowledge of a state space realization of the full order model. This motivates learning reduced models directly from data.

In the data-driven setting, one does not assume access to a realization of the full-order model, but only access to pairs of input-output data. Using provided input-output data, we then learn data-driven reduced order models {(DDROMs)} that well approximate the full order model without ever needing a realization of the full model. {We emphasize that in this paper,  we do not assume access to the samples of the internal state variable. Therefore data-driven frameworks that require sampling the internal state variable, such as \cite{bruntonK2019DataDrivenSEbook, Hokanson2017ProjNonlinLS,KutzBBP2016DMDBook, Moore1981PrincipleComponentAnalysis,QianPW2020Lift&Learn,RajendraB2020DLROM,Schmid2022DMD}, are not considered here.}

We restrict our study to linear time-invariant (LTI) systems.  These systems consist of a system of ordinary differential equations that describes the evolution of the state in response to a time-varying input and a linear mapping from the state to output.  LTI systems arise, e.g., from spatial discretization of partial differential equations.  Such systems may be studied in the time-domain or frequency-domain.  One may switch between these equivalent descriptions by the Laplace transform or Z-transform, for continuous- and discrete-time systems, respectively.

Many successful data-driven techniques exist for constructing DDROMs directly from measured frequency-domain information; see, e.g., \cite{AntoulasBG2020Book,BerlijafaG2017RKFIT,DramacGB2015QuadVF,GustavsenS1998VFOrig,hokansonM2018RationalLS,MayoA2007LoewnerFramework,NakatsukasaST2018AAA,semlyenG1999VF,FlaggBG13}.
Some traditionally intrusive methods have been reformulated in the frequency-domain data-driven framework, see, e.g., \cite{BeattieG2012TFIRKA,GoseaGB2021QuadBT,MayoA2007LoewnerFramework}.
Frequency-domain data-driven techniques have been one of the most commonly used and efficient methods for modeling LTI systems. 
Indeed, interpolation in the frequency domain forms the necessary conditions for $\mathcal H_2$ optimal reduced order modeling~\cite{AntoulasBG2020Book,GugercinAB2008H2}. {Frequency-based techniques require access to frequency-response data, which can be, at times, difficult to obtain in an experimental or simulation setting. This is especially true when re-sampling (active sampling) in the frequency domain is needed, as required in, for example, $\mathcal H_2$ optimal reduced order modeling.  In many situations, one may only have access to time-domain data.}

Then, the question is whether one can extract/infer the necessary frequency-domain data, namely the transfer function values and its derivatives at multiple frequencies, from a single time-domain simulation and thus whether one can mimic the accuracy and efficiency of the frequency-domain techniques from  these inferred frequency-domain data.  In \cite{PeherstorferGW2017TDLow}, based on data from a single time-domain simulation, the authors used the discrete Fourier transform to set up a least-squares problem whose solution approximated the true transfer function samples. Our focus in this paper is on the recent framework by Burohman et al.  \cite{burohmanBSC2020informativity} that established \emph{the necessary and sufficient conditions} to \emph{exactly} recover the frequency-domain data from a given  time-domain simulation.
This powerful methodology of  \cite{burohmanBSC2020informativity}
used the concept of data informativity \cite{vanWaardeECT2023DataInfSurvey,VanWaardeETC2020DataInform}
and amounts to checking two rank conditions and solving a linear system all based on only the given time-domain input/output data.

While data informativity for moment matching {(i.e., interpolation)} \cite{burohmanBSC2020informativity} offers an exciting new framework to recover frequency information from time-domain data, various aspects remain to be resolved for effective use in practical settings. For example, the theoretical framework requires exact rank computations for 
ill-conditioned matrices (due to the associated Hankel-like structure), which will be a rather difficult task numerically. Moreover, the theory  requires solving (potentially large) linear systems that are  ill-conditioned. Also, 
the knowledge of the system order is assumed which in a data-driven setting is undesirable. Our work addresses these issues and provides a robust numerical implementation while providing an error indicator purely based on the given data.
The main contributions of our work are as follows:

\begin{itemize}
    \item \Cref{sec:OrthSubSpace} analytically investigates the conditioning of the resulting linear systems and provides a method for calculating frequency information that does not change theoretical properties of the original solution while addressing numerical ill-conditioning concerns.
    \item \Cref{sec:windowing} motivates the need for and provides an error indicator for the recovered frequency information.
    \item \Cref{sec:calc_n_hat} provides the analysis to remove the assumption that the system order is known.
    \item \Cref{Sec:Algorithm} presents the resulting algorithmic implementation based on  the theoretical and numerical considerations. 
    \item \Cref{sec:Results} evaluates our algorithm on several  reduced order modeling problems, showing that we are able to construct frequency-based ROMs from purely time-domain data.
\end{itemize}
We begin, in \cref{sec:DataInform}, by summarizing key results of \cite{burohmanBSC2020informativity}, which form the foundation of our work. 

\section{The data informativity framework for moment matching}
\label{sec:DataInformativity}
\label{sec:DataInform}
This section summarizes key results from \cite{burohmanBSC2020informativity}, which provides the theoretical framework to recover frequency-response data (transfer function values and its derivatives) of discrete-time linear dynamical systems from time-domain data based on the concept of data informativity~\cite{VanWaardeETC2020DataInform}. 
The background and main results are given here. For proofs and discussion, we refer the reader to the original source~\cite{burohmanBSC2020informativity}. The framework of~\cite{burohmanBSC2020informativity} will form the foundation of our  contributions, discussed in  later sections.

\subsection{{Recovering transfer function values from time-domain data}}
\label{subsec:ProbForm}  
Let $\mathcal S$ be a minimal\footnote{The assumption of minimality is made to simplify the notation in later sections. \cref{prop:DontNeedFulln} will show that this assumption is not indeed needed.},
order $n$  discrete-time single-input-single-output (SISO) system with state-space realization
\begin{equation}
    \label{eq:LinDiscSys}
    \mathcal S: \left \{
    \begin{aligned}
    \mbf x[k+1] &= \mbf A\mbf x[k] + \mbf bu[k]\\
    y[k+1] &= \mbf c^{\top}\mbf x[k],\\
    \end{aligned} \right.
\end{equation}
where $\mbf A \in \bb R^{n \times n}, \,\mbf b \in \bb R^n,\, \mbf c^{\top} \in \bb R^{1\times n}$;
$\mbf x[k] \in \bb R^n$ is the state of \cref{eq:LinDiscSys} at time $k$; $u[k] \in \bb R$ is the input at time $k$; and $y[k]\in \bb R$ is the  output at time  $k$.
{Equivalently, the input/output behavior of the dynamical system \cref{eq:LinDiscSys} can be described via the difference equation 
\begin{align}
     y[t+n] +   p_{n-1}y[t+n-1]&+\cdots  +   p_1y[t+1] +   p_0y[t] = \nonumber \\  & q_n u[t+n] +   q_{n-1}u[t+n-1] + \cdots +   q_1u[t+1] +   q_0u[t],
    \label{eq:diffEQ}
\end{align}
for some coefficient vectors 
\begin{equation}
    \label{eq:pqvecsdef}
    {\mathbf p} = \begin{bmatrix} p_0 & p_1 &\ldots & p_{n-1}\end{bmatrix} \in \mathbb R^{1 \times n}\quad \text{and}\quad {\mathbf q} = \begin{bmatrix} q_0 & q_1 & \ldots &  q_n\end{bmatrix} \in \mathbb R^{1 \times (n+1)}.
\end{equation}
Following~\cite{burohmanBSC2020informativity}, we use the forward shift operator {$z$ defined by $zf[t] = f[t+1]$} to rewrite \cref{eq:diffEQ} as
\begin{equation}
    \label{eq:diffEQShiftOp}
    (z^n + p_{n-1}z^{n-1} + \cdots + p_1z + p_0)y[t] = (q_nz^n + q_{n-1}z^{n-1}+ \cdots+q_1z + q_0)u[t]. 
\end{equation}
Label the polynomials on the left and right sides of \cref{eq:diffEQShiftOp} as
{
\begin{equation}
        P(z) = z^n + p_{n-1}z^{n-1} + \cdots +  p_0~~~\mbox{and}~~~
        Q(z) = q_nz^n + q_{n-1}z^{n-1}+ \cdots+q_0.
\end{equation}
}
Then the transfer function of the dynamical system \cref{eq:LinDiscSys} is
\begin{equation}
    \label{eq:TransFunIsRatioanlExpanded}
    H(z) = \frac{Q(z)}{P(z)} = \frac{q_nz^n + q_{n-1}z^{n-1}+ \cdots+q_1z + q_0}{z^n + p_{n-1}z^{n-1} + \cdots + p_1z + p_0},
\end{equation}
which is an order $n$ rational function in $z$.  Based on  \cref{eq:TransFunIsRatioanlExpanded}, we will call $\mbf p \in \bb R^{1\times n}$ the \emph{denominator coefficient vector} and $\mbf q \in \bb R^{1\times n+1}$ the \emph{numerator coefficient vector}.  Note that if one were to have access to $\mbf A, \mbf b,$ and $\mbf c$ in \cref{eq:LinDiscSys}, then 
\begin{equation}
    H(z) = \mbf c^{\top}(z\mbf I - \mbf A)^{-1}\mbf b.
\end{equation}
The rest of this section outlines how we may go from time-domain input-output data
$$\bb U = \begin{bmatrix}u[0] & u[1] & \ldots & u[T]\end{bmatrix}^{\top} \in \mathbb R^{T+1} \quad \mbox{and} \quad \bb Y = \begin{bmatrix}y[0] & y[1] & \ldots & y[T]\end{bmatrix}^{\top} \in \mathbb R^{T+1}$$
to the values of the transfer function $H(\sigma)$ for given $\sigma \in \bb C.$  Specifically, we assume that the numerator coefficient vector $\mbf p$ and denominator coefficient vector $\mbf q$ are \emph{unknown}, but that the system order $n$ and one time-domain input-output data pair $(\bb U,\bb Y)$ are \emph{known}.  We assume to know $n$ as we follow the discussion of \cite{burohmanBSC2020informativity} in this section.  This assumption will be removed in \cref{sec:calc_n_hat} as one component of the computational and analytical framework that we develop.
}

To begin, assume $T \geq n$, let $t = 0$, and rearrange \cref{eq:diffEQ} as 
\begin{equation}
    \label{eq:diffEqAsInnerProduct}
    \begin{bmatrix}{\mbf q} & -{\mbf p}\end{bmatrix}
    \begin{bmatrix}\bb U(0:n) \\ \bb Y(0:n-1)\end{bmatrix} = y[n] \quad \text{where} \quad {\bb U(0:n) \coloneqq} \begin{bmatrix}u[0] & u[1] & \ldots & u[n]\end{bmatrix}^{\top},
\end{equation}
and similarly for $Y(0:n-1)$. Since $\bb U$ and $\bb Y$ are generated by $\mathcal S$, they must satisfy \cref{eq:diffEQ} for each $t = 0, 1, \ldots, T-n$.  Thus for each $t$, one can express \cref{eq:diffEQ} in the form \cref{eq:diffEqAsInnerProduct}, leading to $T-n+1$ inner products of the form \cref{eq:diffEqAsInnerProduct}, which can be written as the linear system
\begin{equation}
    \begin{bmatrix}{\mathbf q}&-{\mathbf p}\end{bmatrix}\begin{bmatrix}\bb H_n(\bb U)\\\overline{\bb H}_n(\bb Y)\end{bmatrix} = \begin{bmatrix}y[n]&y[n-1]&\ldots&y[T]\end{bmatrix},
    \label{eq:ExplainDataTrue}
\end{equation}
where 
\begin{equation}
\label{eq:hankleMatrixDef}
    \bb H_n(\bb U) = 
    \begin{bmatrix}
        u[0]&u[1]&\ldots &u[{T-n}]\\
        u[1]&u[2]&\ldots&u[{T-n+1}]\\
        \vdots&\vdots &\ddots & \vdots \\
        u[n] & u[{n+1}] & \ldots & u[T]
    \end{bmatrix} 
    \in \bb R^{(n+1) \times (T-n + 1)}
\end{equation}
is the Hankel matrix of depth $n$ (and similarly for $\bb H_n(\bb Y)$), and $\overline{\bb H}_n(\bb Y)$ $\in \bb R^{n\times (T-n+1)}$ is $\bb H_n(\bb Y)$ with the last row removed. Now  one can define the condition for when a system 
with transfer function $\widetilde H(z)$ having the numerator and denominator coefficient  vectors
$\tilde{\mbf q}$ and $\tilde{\mbf p}$ as in \cref{eq:pqvecsdef}
can explain the data in ($\bb U,\bb Y$).
\begin{definition}
    An order $n$ system 
    with numerator and denominator coefficient vectors $\tilde{\mbf q}$ and $\tilde{\mbf p}$ as in~\cref{eq:pqvecsdef} can \emph{explain the data} $(\bb U,\bb Y)$ if
    \begin{equation}
    \begin{bmatrix}\tilde{\mathbf q}&-\tilde{\mathbf p}\end{bmatrix}\begin{bmatrix}\bb H_n(\bb U)\\\overline{\bb H}_n(\bb Y)\end{bmatrix} = \begin{bmatrix}y[n]&y[n-1]&\ldots&y[T]\end{bmatrix}.
    \label{eq:ExplainData}
\end{equation}
\end{definition}
In other words, the coefficient vectors $\tilde{\mbf p}$ and $\tilde{\mbf q}$ can \emph{explain the data} in $(\bb U,\bb Y)$ if a system with the coefficient vectors $\tilde{\mbf p}$ and $\tilde{\mbf q}$ produces the output $\bb Y$ when driven by the input $\bb U$.
The set of all coefficient vectors $\begin{bmatrix}\tilde{\mathbf q}&-\tilde{\mathbf p}\end{bmatrix} \in \bb R^{2n+1}$ (corresponding to order $n$ systems) that can explain the data in $(\bb U,\bb Y)$ is
denoted by 
\begin{equation}
\Sigma^{n}_{(\bb U,\bb Y)} = \left \{\begin{bmatrix}\tilde{\mathbf q}&-\tilde{\mathbf p}\end{bmatrix} \in \bb R^{1 \times (2n+1)} \mid \text{ \cref{eq:ExplainData} holds}\right \}.
\label{set:ExplainData}
\end{equation}

Following \cite{burohmanBSC2020informativity}, we call $ H(\sigma)$ the {\emph{zeroth moment}}  of  $\mathcal S$ at $\sigma$ and write $H(\sigma) = M_0$ (or $H(\sigma) = M_0(\sigma)$ when we wish to emphasize $\sigma$). Similar to \cref{eq:ExplainDataTrue}, one can rewrite \cref{eq:TransFunIsRatioanlExpanded} evaluated at $z = \sigma$ as a linear system
\begin{equation}
    \begin{bmatrix}{\mathbf q} &-{\mathbf p}\end{bmatrix}\begin{bmatrix}\gamma_n(\sigma)\\M_0 \gamma_{n-1}(\sigma)\end{bmatrix} = M_0 \sigma^n,
    \label{eq:0MomentTrue}
\end{equation}
where
$\gamma_k(\sigma) = \begin{bmatrix}1 & \sigma & \sigma^2 & \ldots & \sigma^k\end{bmatrix}^{\top}  \in \bb C^{k+1}.$
Then,~\cref{eq:0MomentTrue} leads to the condition under which a system with coefficient vectors $\tilde{\mbf q}$ and $\tilde{\mbf p}$ has zeroth moment $M_0$ at $\sigma$.
{
\begin{definition} An order $n$ transfer function $\widetilde H(z)$ with numerator and denominator coefficient  vectors $\tilde{\mbf q}$ and $\tilde{\mbf p}$ as in~\cref{eq:pqvecsdef} has \emph{zeroth moment $M_0$ at $\sigma$} if
     \begin{equation}
    \begin{bmatrix}\tilde{\mathbf q} &-\tilde{\mathbf p}\end{bmatrix}\begin{bmatrix}\gamma_n(\sigma)\\M_0 \gamma_{n-1}(\sigma)\end{bmatrix} = M_0 \sigma^n.
    \label{eq:0Moment}
\end{equation}
\end{definition}
}
Then the set of all order $n$ systems with zeroth moment $M_0$ at $\sigma$ is
\begin{equation}
\Sigma^{n,0}_{\sigma,M_0} = \left \{\begin{bmatrix}\tilde{\mathbf q}&-\tilde{\mathbf p}\end{bmatrix} \in \bb R^{1 \times (2n+1)} \mid \text{ \cref{eq:0Moment} holds}\right \}
\label{set:0Moment}.
\end{equation}
To find the theoretical conditions under which one can recover a unique transfer function value $H(\sigma) = M_0$ from input-output data $(\bb U,\bb Y)$, \cite{burohmanBSC2020informativity}~requires that all systems that explain the data $(\bb U,\bb Y)$ must also have transfer function value $M_0$ at $\sigma$.
As stated earlier, the goal in~\cite{burohmanBSC2020informativity} is not to fully identify the underlying system $\mathcal S$. Instead, 
 given a frequency sampling point $\sigma \in \mathbb C$, 
the goal is to determine whether the data is informative enough to recover $H(\sigma)$, the transfer function value at 
$z=\sigma$.
\begin{definition}
    \label{def:informForinterp}
    The data $(\bb U,\bb Y)$ are \emph{informative} for {interpolation (moment matching)} at $\sigma$ if there exists a unique $M_0 \in \bb C$ such that
    \begin{equation}
        \label{eq:CanInterpolateDef}
        \Sigma^n_{(\bb U,\bb Y)} \subseteq \Sigma_{\sigma, M_0}^{n,0}.
    \end{equation}
\end{definition}
Leveraging \cref{eq:0Moment} and \cref{eq:ExplainData} and applying \cref{def:informForinterp}, one is presented with a method for calculating $M_0$.
\begin{theorem}[{\cite[Lem.~10 \& Thm.~12]{burohmanBSC2020informativity}}]
    \label{thm:ExistUniqueOrig}
    Assume access to input data $\bb U$ and output data $\bb Y$.  Let $n$ be the order of the system and let $\sigma, M_0 \in \mathbb C$.  Then the inclusion $\Sigma^n_{(\bb U,\bb Y)} \subseteq \Sigma_{\sigma,M_0}^{n,0}$ holds if and only if there exists $\xi \in \mathbb C^{T-n+1}$ such that
    \begin{equation}
        \begin{bmatrix}\bb H_n(\bb U)&0\\\bb H_n(\bb Y)&-\gamma_n(\sigma)\end{bmatrix}
        \begin{bmatrix}\xi\\M_0\end{bmatrix}
        = \begin{bmatrix}\gamma_n(\sigma)\\0\end{bmatrix}.
        \label{eq:calcM0Orig}
    \end{equation}  
Moreover, 
     the data $(\bb U$,$\bb Y)$ are informative for interpolation at $\sigma$ (i.e., the inclusion $\Sigma^n_{(\bb U,\bb Y)} \subseteq \Sigma_{\sigma,M_0}^{n,0}$ holds with a unique $M_0$) if and only if
    \begin{equation}
        \label{eq:ExistCondOrig}
        \rank 
        \begin{bmatrix}\bb H_n(\bb U) & 0 & \gamma_n(\sigma)\\
            \bb H_n(\bb Y) & \gamma_n(\sigma) & 0
        \end{bmatrix}
        =
        \rank 
        \begin{bmatrix}
            \bb H_n(\bb U)&0\\\bb H_n(\bb Y)& \gamma_n(\sigma)
        \end{bmatrix}
    \end{equation}
    and
    \begin{equation}
        \rank \begin{bmatrix}\bb H_n(\bb U)&0\\\bb H_n(\bb Y)& \gamma_n(\sigma)\end{bmatrix}
        =
        \rank \begin{bmatrix}\bb H_n(\bb U) \\ \bb H_n(\bb Y)\end{bmatrix} + 1.
        \label{eq:UniqueCondOrig}
    \end{equation}
\end{theorem}
The existence of $M_0$ such that \cref{eq:calcM0Orig} holds is equivalent to \cref{eq:ExistCondOrig}, which guarantees that the right-hand side of \cref{eq:calcM0Orig} is in the range of the matrix on the left-hand side.  However, the informativity for interpolation at $\sigma$ requires the uniqueness of $M_0$. This uniqueness condition is equivalent \cref{eq:UniqueCondOrig}.

{
\begin{remark}
    For frequency-based modeling / model reduction of linear dynamics, it is advantageous to have access to the first derivative of the transfer function at sampling frequencies. Indeed, optimal model reduction in the $\mathcal{H}_2$ norm requires not only interpolating  transfer function values at some selected points but also its derivatives, see,~e.g.,~\cite{AntoulasBG2020Book, GugercinAB2008H2}.  
    Assuming the data $(\bb U,\bb Y)$ are informative for interpolation at $\sigma \in \bb C$, the data is informative for \emph{Hermite interpolation} (or interpolation of $H'$) at $\sigma$ if and only if there exists $\xi \in \bb C^{T-n+1}$ such that
    \begin{equation}
        \begin{bmatrix}\bb H_n(\bb U)&0\\\bb H_n(\bb Y)&-\gamma_n(\sigma)\end{bmatrix}
        \begin{bmatrix}\xi\\M_1\end{bmatrix}
        = \begin{bmatrix}\gamma^{(1)}_n(\sigma)\\M_0\gamma_n^{(1)}(\sigma)\end{bmatrix}
        \label{eq:calcMkOrig}
    \end{equation}
    where
    \begin{equation}  \label{eq:gamma1}
        \gamma^{(1)}_k(\sigma) = \begin{bmatrix}0 & 1 & 2\sigma & \ldots & k\sigma^{k-1}\end{bmatrix}^{\top} \in \bb C^{k+1}
    \end{equation}
    holds for a unique $M_1$.  Moreover, the data is informative for Hermite interpolation at $\sigma$ if and only if the data is informative for interpolation at $\sigma$ and
    \begin{equation}
        \rank \begin{bmatrix}\bb H_n(\bb U) & 0 & \gamma^{(1)}_n(\sigma)\\
         \bb H_n(\bb Y) & \gamma_n(\sigma) & M_0\gamma^{(1)}_n(\sigma)\end{bmatrix}
         =
         \rank \begin{bmatrix}\bb H_n(\bb U)&0\\\bb H_n(\bb Y)& \gamma_n(\sigma)\end{bmatrix}.
    \label{eq:ExistCondMk}
    \end{equation}
    Note that there is only one rank condition (for existence) in this case, as assuming the data is informative for interpolation provides the uniqueness condition.  Here we have only presented the final result but the derivation follows much the same as the interpolation case; for details, see \cite{burohmanBSC2020informativity}.
\end{remark}
}

\section{{Theoretical and algorithmic analysis for 
a robust implementation of data informativity for moment matching}}
\label{sec:Implementation}

The results presented in the previous section, taken from \cite{burohmanBSC2020informativity}, give us a powerful theoretical framework to recover frequency information $M_0=H(\sigma)$ from time-domain input data $\bb U$ and output data $\bb Y$ (in exact arithmetic and under the additional assumption that the system order $n$ is known.) One first checks the conditions \cref{eq:ExistCondOrig} and \cref{eq:UniqueCondOrig} and then solves the linear system given in \cref{eq:calcM0Orig} to obtain $M_0$. However, in a computational setting, key issues remain.  First of all, the presence of the Hankel matrices $\bb H_n(\bb U)$ and $\bb H_n(\bb Y)$
is expected to yield  ill-conditioned linear systems since Hankel matrices are usually ill-conditioned
\cite{Florian2010HankCondNum,Beckermann2000HankCondNum,BeckermannT2019SValDisplacement}. Additionally (and especially due to this ill-conditioning), the rank decisions to check the conditions \cref{eq:ExistCondOrig} and \cref{eq:UniqueCondOrig} 
(and the corresponding subspace computations we introduce later)
are expected to be numerically ambiguous. Due to the potentially imprecise rank decisions (imprecise/truncated subspaces) and ill-conditioned linear systems to solve combined with the fact the input/output data is most likely noisy, 
in practice we expect $M_0$ (recovered from data) to \emph{approximate} but not exactly equal $H(\sigma)$.
These observations are the main motivations for the analysis and computational framework we develop in the rest of the paper.  We use the convention that $M_0 \in \bb C$ is the approximation to $H(\sigma)$ recovered by solving \cref{eq:calcM0Orig}.

{
We first investigate (and demonstrate) the ill-conditioning of the linear system \cref{eq:calcM0Orig} 
in \cref{sec:orthogonalization} and provide methods to overcome it.  In \cref{sec:RankCondsAsMatVecs}, we revisit the rank conditions \cref{eq:ExistCondOrig} and \cref{eq:UniqueCondOrig}, and provide alternatives that are {faster to check} and numerically easier to implement.  Although this analysis and the resulting modification significantly improves  ill-conditioning, in a computational setting involving various numerical rank decisions, one still expects some error in recovered transfer function values. Therefore, to test the accuracy of the recovered values, in \cref{sec:windowing} we provide an error indicator based on a windowing strategy.  In \cref{sec:calc_n_hat} we show that exact knowledge of the system order $n$ is not required to recover $M_0 \approx H(\sigma)$ with good accuracy.  We then put the whole analysis and computational framework together in \cref{Sec:Algorithm} to provide a robust numerical algorithm. Even though
the discussion in this section is limited to recovery of the value of the transfer function at a given point,  all techniques and analyses extend directly to recovery of derivative information.
}

\subsection{Improving and predicting conditioning}
\label{sec:OrthSubSpace}
As stated before, the presence of Hankel matrices in the linear system~\cref{eq:calcM0Orig} is expected to lead to ill-conditioning and potential loss of accuracy in the recovery of the transfer functions values.  In \cref{sec:orthogonalization} we provide a (typically) much better conditioned linear system whose solution is theoretically equivalent to that of \cref{eq:calcM0Orig} and thus still allows for recovery of $H(\sigma)$ from $(\bb U,\bb Y)$.  This new linear system also leads to a faster and more robust method of checking the rank conditions \cref{eq:ExistCondOrig} and \cref{eq:UniqueCondOrig}, which is explored in \cref{sec:RankCondsAsMatVecs}.

For brevity in the coming sections, we introduce the following notation:
\begin{center}
  \begin{subequations}
    \parbox[c]{\textwidth*7/30}{
    \begin{equation}
        \label{eq:GnDef}
        \mbf G_n \coloneqq \bmat{\bb H_n(\bb U)\\ \bb H_n(\bb Y)},
    \end{equation}
    }~\hspace{0.19in}
    \parbox[c]{\textwidth*10/30}{
    \begin{equation}
        \label{eq:zSigmaDef}
        \mbf z(\sigma) \coloneqq \bmat{\mbf 0^{\top} & -\gamma_n(\sigma)^{\top}}^{\top},
    \end{equation}
    }~\hspace{0.19in}
    \parbox[c]{\textwidth*9/30}{
    \begin{equation}
        \label{eq:bSigmaDef}
        \mbf b(\sigma) \coloneqq \bmat{\gamma_n(\sigma)^{\top} & \mbf 0^{\top}}^{\top}.
    \end{equation}
    }
\end{subequations}  
\end{center}

\subsubsection{Orthogonalization}
\label{sec:orthogonalization}
The coefficient matrices in linear systems~\cref{eq:calcM0Orig} to be solved are composed of Hankel matrices, which are typically ill-conditioned \cite{Florian2010HankCondNum,Beckermann2000HankCondNum,BeckermannT2019SValDisplacement}. 
\cref{prop:GeneralOrthSubspaceWorks} and \cref{cor:OrthSubspaceOK} below provide a way to generate linear systems that typically have much lower condition numbers than \cref{eq:calcM0Orig}, while maintaining the same theoretical solution for $M_{0}$.

\begin{lemma}
    \label{prop:GeneralOrthSubspaceWorks}
    Let $\mbf G \in \bb C^{s \times r}$, $\mbf z \in \bb C^s$, and 
$\mbf b \in \bb C^s$ satisfy
    \begin{center}
      \begin{subequations}
        \parbox[c]{\textwidth*33/60}{
        \begin{equation}
            \label{eq:existCondGeneral}
            \rank\bmat{\mbf G & \mbf z & \mbf b} = \rank\bmat{\mbf G & \mbf z}
            \quad \text{and}
        \end{equation}
        }
        \parbox[c]{\textwidth*26/60}{
        \begin{equation}
            \label{eq:uniqueCondGeneral}
            \rank\bmat{\mbf G & \mbf z} = \rank\bmat{\mbf G} + 1.
        \end{equation}
        }
    \end{subequations}  
    \end{center}
    Let 
    \begin{equation}
        \label{eq:SVDofGeneralSys}
            \mbf U\Sigma \mbf V^* = \mbf G
    \end{equation}
    be the short SVD of $\mbf G,$ i.e., for some $p \leq \min(s,r)$, $\mbf U \in \bb R^{s\times p}$ and $\mbf V \in \bb R^{r \times p}$ have orthonormal columns and $\Sigma \in \bb R^{p\times p}$ is diagonal with strictly positive diagonal entries.  Then there exist $\mbf x \in \bb C^{r+1}$ and $\hat{\mbf x} \in \bb C^{p+1}$ such that
    \begin{center}
    \begin{subequations}
    \parbox[l]{\textwidth*10/30}{
        \begin{equation}
            \label{eq:GeneralGsystem}
            \bmat{\mbf G & \mbf z} \mbf x = \mbf b
        \end{equation}
    }~\hspace{0.3in}
    \mbox{and}~\hspace{0.3in}
    \parbox[l]{\textwidth*10/30}{
        \begin{equation}
            \label{eq:GeneralUSystem}
            \bmat{\mbf U & \mbf z}\hat{\mbf x} = \mbf b.
        \end{equation}
    }
    \end{subequations}
    \end{center}
    Further, it must be that the last components of $\mbf x$ and $\hat{\mbf x}$ are equal and unique.
\end{lemma}
\begin{proof}
    Existence of $\mbf x$ satisfying 
    \cref{eq:GeneralGsystem}
    follows from \cref{eq:existCondGeneral} and uniqueness of its last component, $\mbf x_{r+1}$, from \cref{eq:uniqueCondGeneral}. Then using \cref{eq:SVDofGeneralSys}, we rearrange \cref{eq:GeneralGsystem} as
    \begin{equation}
        \bmat{\mbf G & \mbf z}\mbf x =  \bmat{\mbf U\Sigma \mbf V^* & \mbf z}\mbf x = \bmat{\mbf U & \mbf z}\bmat{\mbf{\Sigma}\mbf V^*\mbf x(1\!:\!r)\\\mbf x_{r+1}} = \mbf b,
    \end{equation}
    where $\mbf x(1\!:\!r) \in \bb C^{r}$ denotes the first $r$ entries of $\mbf x$. Thus
    $\hat{\mbf x} = \sbmat{\mbf{\Sigma}\mbf V^*\mbf x(1:r)\\\mbf x_{r+1}} \in \bb C^{p+1}$
    is a solution to \cref{eq:GeneralUSystem} and satisfies $\hat{\mbf x}_{p+1} = \mbf x_{r+1}$. i.e., the last components of $\mbf x$ and $\hat{\mbf x}$ are equal.  Since the matrix $\mbf U$ is an orthonormal basis for the range of $\mbf G$ and (by \cref{eq:uniqueCondGeneral}) $\mbf z$ is not in the range of $\mbf G$, the matrix $\bmat{\mbf U & \mbf z}$
    is full rank, which means $\hat{\mbf x}$ is unique and in particular, $\hat{\mbf x}_{p+1}$ must be unique.
\end{proof}
\cref{prop:GeneralOrthSubspaceWorks} can be specialized to our setting of recovering frequency information $H(\sigma)$ from time-domain data $(\bb U,\bb Y)$ via \cref{cor:OrthSubspaceOK}. 
\begin{corollary}
    \label{cor:OrthSubspaceOK}
    Assume access to time-domain data $(\bb U,\bb Y)$ generated by an order $n$ system $\mathcal S$ as in \cref{eq:LinDiscSys} with transfer function $H(z)$.    Let $\mbf G_{n}, \mbf z(\sigma),$ and $\mbf b(\sigma)$ be constructed as in \cref{eq:GnDef}, \cref{eq:zSigmaDef}, and \cref{eq:bSigmaDef}, respectively.  Assume for $\sigma \in \bb C,$ the rank conditions \cref{eq:ExistCondOrig} and \cref{eq:UniqueCondOrig} are satisfied.  Let
    \begin{equation}
        \mbf U\Sigma \mbf V^* = \mbf G_{n},
    \end{equation}
    where $\mbf U \in \bb R^{2(n+1)\times p}, \Sigma \in \bb R^{p \times p},$ and $\mbf V^* \in \bb R^{p \times (T-n+1)}$ for $p \leq T- n + 1$ be the thin SVD of $\mbf G_{n}$.  Then there exists $\xi \in \bb C^{T-n}$ and $\hat{\xi} \in \bb C^{p}$ such that
    \begin{center}
    \begin{subequations}
    \parbox[l]{\textwidth*14/30}{
        \begin{equation}
            \label{eq:CalcM0InOrthSubProof}
            \bmat{\mbf G_{n} & \mbf z(\sigma)}\bmat{\xi \\ H(\sigma)} = \mbf b(\sigma)
        \end{equation}
    }
    \parbox[l]{\textwidth*14/30}{
        \begin{equation}
            \label{eq:CalcM0OrthInOrthSubProof}
            \bmat{\mbf U & \mbf z(\sigma)}\bmat{\hat{\xi} \\ H(\sigma)} = \mbf b(\sigma).
        \end{equation}
    }
    \end{subequations} 
    \end{center}
    Therefore, the last components of any solution to \cref{eq:CalcM0InOrthSubProof} and \cref{eq:CalcM0OrthInOrthSubProof}
    is $H(\sigma).$
\end{corollary}
The proof of \cref{cor:OrthSubspaceOK} follows from \cref{prop:GeneralOrthSubspaceWorks}.  \cref{prop:GeneralOrthSubspaceWorks} leads to one more corollary, which allows us to also rewrite \cref{thm:ExistUniqueOrig} using an orthogonal basis.
\begin{corollary}
    \label{cor:ExistUniqueCondOrth}
    Let $\mbf U = \orth(\mbf G_{n})$
    be an orthonormal basis for the range of $\mbf G_{n}$. 
    Then the data $(\bb U$,$\bb Y)$ are informative for interpolation at $\sigma$ if and only if 
    \begin{subequations}
    \begin{equation}
        \label{eq:ExistCondOrth}
        \rank\left(\bmat{\mbf U & \mbf z(\sigma) & \mbf b(\sigma)}\right) = \rank\left(\bmat{\mbf U & \mbf z(\sigma)}\right),~\mbox{and}
            \end{equation}
                \begin{equation}
        \label{eq:UniqueCondOrth}
        \rank\left(\bmat{\mbf U & \mbf z(\sigma)}\right) = \rank\left({\mbf U}\right) + 1.
    \end{equation}
    \end{subequations}
\end{corollary}
\cref{cor:ExistUniqueCondOrth} will be further exploited in \cref{sec:RankCondsAsMatVecs}, where we show how the conditions \cref{eq:ExistCondOrth} and \cref{eq:UniqueCondOrth} may be checked quickly and robustly.

We emphasize that the implication of \cref{cor:OrthSubspaceOK} is much different than simply using the SVD of the coefficient matrix to solve a linear system. In other words, \cref{cor:OrthSubspaceOK} is \emph{not} suggesting to simply use the SVD of $\bmat{\mbf G_{n} & \mbf z(\sigma)}$. To show the distinction more clearly, let
$\bmat{\mbf G_{n} & \mbf z(\sigma)} = \hat{\mbf U}\hat{\Sigma}\hat{\mbf V}^*$, {where}
        $\hat{\mbf U} \in \bb C^{2(n+1)\times (p+1)}, \hat{\Sigma} \in \bb C^{(p+1)\times (p+1)},$ and $\hat{\mbf V} \in \bb C^{(T-n +1) \times (p+1)}$
be the thin SVD of $\bmat{\mbf G_{n} & \mbf z(\sigma)}$.  Then using this SVD, one can solve \cref{eq:calcM0Orig} by
$
\mbf x 
=
\hat{\mbf V} \hat{\Sigma}^{-1} \hat{\mbf U}^*
\mbf b(\sigma)$ and recover the zeroth moment by setting $M_0 = {\mbf x}_{p+1}$, i.e., the last entry of $\mbf x$.
However, this method of solving for $M_0$ does not improve the conditioning  as the ill-conditioning in $\bmat{\mbf G_{n} & \mbf z(\sigma)}$
is still present  as ill conditioning in $\hat{\Sigma}$. 
In contrast, exploiting the fact that all we need is the last component of $\mbf x$, not the whole solution,  \cref{cor:OrthSubspaceOK} does \emph{not} require us to keep the singular values $\hat{\Sigma}$ of $\mbf G_n$ at all. This allows us to work with a much better conditioned linear system~\eqref{eq:CalcM0OrthInOrthSubProof} as opposed to~\eqref{eq:CalcM0InOrthSubProof}. {In other words, the ill-conditioning of $\mbf G_n$ encoded in  $\hat{\Sigma}$ is  avoided.}

\begin{figure}
    \centering
    \includegraphics[trim = {0 0 0 .5cm},clip,width = .8\textwidth]{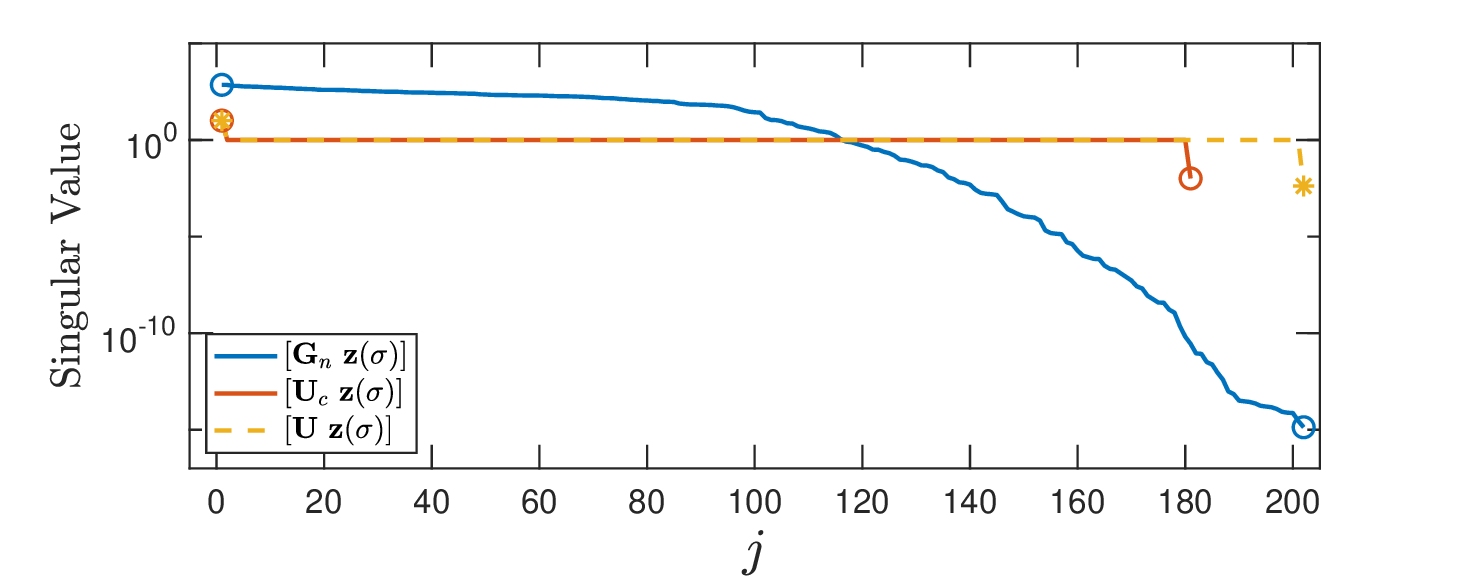}
    \caption{Singular values of
    $\bmat{\mbf G_{n} & \mbf z(\sigma)}$, $\bmat{\mbf U_c & \mbf z(\sigma)}$, and $\bmat{\mbf U & \mbf z(\sigma)}$.}
    \label{fig:SingularValueCompare}
\end{figure}

Not surprisingly, we anticipate the matrix $\bmat{\mbf U & \mbf z(\sigma)}$ to be much better conditioned than $\bmat{\mbf G_n & \mbf z(\sigma)}$. To illustrate this point numerically, we construct a discrete-time, stable linear dynamical system (as in $\cref{eq:LinDiscSys}$) denoted by $\mathcal S_0$ with transfer function $H_0(z)$. The system $\mc S_0$ has order $n = 100$ with poles randomly placed in the unit disc and with random residues.  We simulate $\mathcal S_0$ for $T = 3n = 300$ time steps with a Gaussian random input $\bb U \in \bb R^{T+1}$ to obtain the corresponding output $\bb Y \in \bb R^{T+1}$ and form $\mbf G_n$ as in \cref{eq:GnDef} and $\mbf U = \orth(\mbf G_n)$. 

\cref{cor:OrthSubspaceOK} assumes that $\mbf U$ exactly spans the range of ${\mbf G}_n$, i.e., no columns corresponding to non-zero singular values are truncated. But, in practice,  the number of columns of $\mbf U$ will be determined by a decision about the rank of $\mbf G_n$. Therefore, in addition to computing $\mbf U$ with the same number of columns as ${\mbf G}_n$, which is \emph{theoretically} a full-rank matrix in this example, we also compute a truncated SVD (and thus an approximate basis for the range of ${\mathbf G}_n$) where the \emph{numerical} rank of ${\mathbf G}_n$ is computed based on a given tolerance (we used the \MATLAB's numerical rank decision in the 
built-in {\tt rank} command). 
We call the resulting \emph{computed} basis ${\mbf U}_{c}$. In this example,
while $\mbf G$ and $\mbf U$ have $201$ columns, ${\mbf U}_{c}$ has 180 columns.
We pick $\sigma = e^{0.5\mbf i}$ {where $\mbf i^2 = -1$}  and plot  the singular values of the coefficient matrices $\bmat{\mbf G_n & \mbf z(\sigma)}$,  $\bmat{\mbf U & \mbf z(\sigma)}$, and $\bmat{{\mbf U}_c & \mbf z(\sigma)}$
in \cref{fig:SingularValueCompare}.   
As expected, the condition numbers of $\bmat{\mbf U & \mbf z(\sigma)}$ and $\bmat{\mbf U_c & \mbf z(\sigma)}$ are much smaller than that of $\bmat{\mbf G_{n} & \mbf z(\sigma)}$. More precisely, we have $\kappa_2(\bmat{\mbf U & \mbf z(\sigma)}) = 3.75\times 10^3$, $\kappa_2(\bmat{\mbf U_c & \mbf z(\sigma)}) = 9.93\times 10^2$, and $\kappa_2(\bmat{\mbf G_{n} & \mbf z(\sigma)}) = 4.58\times 10^{17}$
where $\kappa_2(\cdot)$ denotes  the $2$-norm condition number. 
\begin{remark}
    The condition number of the matrix $\bmat{\mbf G_n & \mbf z(\sigma)}$,  $\kappa_2(\bmat{\mbf G_n & \mbf z(\sigma)})$ 
    is a key quantity in the perturbation analysis of the linear system  $\bmat{\mbf G_n & \mbf z(\sigma)}\mbf x = \mbf b$ and provides an upper bound for the relative error in the solution for the perturbed system. However, this is a worst-case scenario upper bound and for some right-hand sides $\mbf b$, the upper bound predicted by the condition number might be pessimistic. Indeed, for the specific problem in the above numerical example, this is the case and the value of $M_0$ obtained from solving the original problem \cref{eq:calcM0Orig} explicitly still provides an accurate approximation to $H(\sigma)$. However, the advantage of using $\mbf U$ (or $\mbf U_c$) is that the problem to be solved  will be well conditioned for \emph{any} right-hand side $\mbf b$, which is a much more desirable scenario to be in since we would like to estimate $H(\sigma)$ for many different values of $\sigma$. We will present additional advantages next.
\end{remark}

Recall that the data-informativity framework we consider here requires checking the rank conditions \cref{eq:ExistCondOrig} and \cref{eq:UniqueCondOrig}.  The rank calculation is an ambiguous task in general, especially when there is not a clear cut-off/gap in the singular values. This is precisely the case for the singular values of $\bmat{\mbf G_{n} & \mbf z(\sigma)}$, which decay gradually to machine precision; making the rank decision difficult. On the other hand, we see that finding the rank of $\bmat{\mbf U & \mbf z(\sigma)}$ and $\bmat{\mbf U_c & \mbf z(\sigma)}$ are much less ambiguous. We revisit these issues in more detail in the next section.

In data-driven reduced-order modeling methods that use frequency-domain data, one would require access to
many frequency samples $H(\sigma_i), i=1,2,\ldots,M$, where $M$ could  be potentially large. Thus, one will have to solve the linear system~\cref{eq:CalcM0OrthInOrthSubProof} multiple times for different $\mbf z(\sigma)$ and $\mbf b(\sigma)$ values. However, the orthogonal basis $\mbf U = \orth(\mbf G_n)$ in~\cref{eq:CalcM0OrthInOrthSubProof} is \emph{fixed} (since it only depends on the input/output pairs not the frequency $\sigma$) and can be precomputed. Then solving ~\cref{eq:CalcM0OrthInOrthSubProof} can be achieved with only, e.g., one step of orthogonalization (since $\mbf U$ already has orthonormal columns) for different frequencies. Thus, working with a precomputed  $\mbf U$ 
also drastically reduces the computational cost.

\subsubsection{Revisiting rank calculations}
\label{sec:RankCondsAsMatVecs}
As we pointed out earlier, in a computational setting we will be working with the computed basis ${\mbf U}_c$. Thus, added to the fact that numerical rank decisions are ambiguous in general,
we should not expect to be able to exactly satisfy (or check) the rank conditions \cref{eq:ExistCondOrth}
and \cref{eq:UniqueCondOrth} in practice (which are equivalent to the original rank conditions \cref{eq:ExistCondOrig} and \cref{eq:UniqueCondOrig} by \cref{cor:ExistUniqueCondOrth}). Thus, 
we outline a procedure for checking both \cref{eq:ExistCondOrth} and \cref{eq:UniqueCondOrth} to a user-defined tolerance in this section. 
{Even though replacing  $\bmat{\mbf G_{n} & \mbf z(\sigma)}$ with  $\bmat{\mbf U_c & \mbf z(\sigma)}$ has made the rank decisions clear 
for the example above as shown in \cref{fig:SingularValueCompare}
(and potentially in most cases), we still want to revisit these rank decisions to increase robustness for all cases.}
We present the  analysis for the true basis $\mbf U$.

Since  $\mbf U$ has orthonormal columns (and full column rank),  condition \cref{eq:UniqueCondOrth} is satisfied if and only if
$\mbf z(\sigma) \not\in \range(\mbf U).$
This is equivalent to
\begin{equation}
    \label{eq:UniqueCondMatVecExact}
    \|(\mbf I - \mbf U\mbf U^{H})\mbf z(\sigma)\| > 0.
\end{equation}
Due to the machine-precision arithmetic (and the truncated subspace computations), we expect that \cref{eq:UniqueCondMatVecExact} will always be satisfied. Instead, we pick a tolerance $\tau_1$ and say that $\mbf z(\sigma) \not\in \range(\mbf U)$ if
\begin{equation}
    \label{eq:UniqueCondMatVec}
    \|(\mbf I - \mbf U\mbf U^{H})\mbf z(\sigma)\| \geq \tau_1\|\mbf z(\sigma)\|.
\end{equation}
The condition \cref{eq:UniqueCondMatVec} is easy to check and only involves matrix-vector multiplications (and norm computation). We understand that the new condition has its own ambiguity as it requires choosing a tolerance $\tau_1$ to decide if a quantity is small enough or not.
However, the ambiguity in \cref{eq:UniqueCondMatVec} is easier to quantify than deciding if the rank of ${\mbf G}_n$ increases by one or not when appended by $\mbf z(\sigma)$, especially when the singular value decay is as shown in  \cref{fig:SingularValueCompare}.

We now express \cref{eq:ExistCondOrth} in a similar fashion.
Note that \cref{eq:ExistCondOrth} holds if and only if
\begin{equation}
    \label{eq:ExistCondRange0}
    \mbf b(\sigma) \in \range\bmat{\mbf U & \mbf z(\sigma)}.
\end{equation}
To examine \cref{eq:ExistCondRange0}, we first define the vectors
$$\mbf v = (\mbf I - \mbf U\mbf U^H)\mbf z(\sigma) \quad \text{and} \quad \mbf b^{\perp} = (\mbf I - \mbf U\mbf U^H)\mbf b(\sigma).$$
Now \cref{eq:ExistCondRange0} holds if and only if
\begin{equation}
    \label{eq:ExistCondRange}
    \mbf b^{\perp} \in \range(\mbf v).
\end{equation}
Note that in exact arithmetic \cref{eq:ExistCondRange} also covers the case where $\mbf b(\sigma) \in \range(\mbf U)$ {since then $\mbf b^{\perp} = 0$}.  Then \cref{eq:ExistCondRange} can be expressed as
\begin{equation}
    \label{eq:ExistCondMatVecExact}
    \left\|\left(\mbf I - \frac{\mbf v\mbf v^H}{\|\mbf v\|^2}\right)\mbf b^{\perp}\right\| = 0.
\end{equation}
Then, as before, we introduce a second tolerance $\tau_2$ and consider \cref{eq:ExistCondMatVecExact} satisfied if
\begin{equation}
    \label{eq:ExistCondMatVec}
    \left\|\left(\mbf I - \frac{\mbf v\mbf v^H}{\|\mbf v\|^2}\right)\mbf b^{\perp}\right\| \leq \tau_2\|\mbf b(\sigma)\|.
\end{equation}
The main cost of \cref{eq:ExistCondMatVec} is to compute matrix-vector products.
The user is able to set the parameters $\tau_1$ and $\tau_2$ to determine if the data is informative for interpolation.
We find in our numerical examples that using $\tau_1 = \tau_2 = 10^{-10}$ works effectively. However, by no means these are  optimal choices and the user can vary these tolerances.

We note that when $\tau_1 = 0$ and $\tau_2 = 0,$ \cref{eq:ExistCondMatVec} and \cref{eq:UniqueCondMatVec} are theoretically equivalent to  \cref{eq:ExistCondOrig} and \cref{eq:UniqueCondOrig}.  {Whenever $\tau_1 > 0$ in \cref{eq:UniqueCondMatVec}, we are actually enforcing a stricter requirement than \cref{eq:UniqueCondOrig}.  That is, any vector $\mbf z(\sigma)$ that satisfies \cref{eq:UniqueCondMatVec} will also satisfy \cref{eq:UniqueCondMatVecExact}, which is equivalent to \cref{eq:UniqueCondOrig}.}
Whenever $\tau_2 > 0$,
{as opposed to the linear system~\cref{eq:CalcM0OrthInOrthSubProof},}
we are solving the \emph{least squares problem}
\begin{equation}
    \label{eq:LSSolution}
    \argmin_{\mbf x \in \bb R^{p+1}}
    \left\|\begin{bmatrix}
        \mbf U
        & \mbf z(\sigma)
    \end{bmatrix}
    \mbf x
    -
    \mbf b(\sigma)
    \right\|; \quad \mbox{and~then~setting}~M_{0} = \mbf x_{p+1}.
\end{equation}
Allowing recovery of $M_0$ via a least squares problem is natural in this setting, especially if the input/output data were to be obtained experimentally.  As already mentioned, we will be using a computed, and potentially truncated, basis ${\mbf U}_c$ in place of $\mbf U$ that \cref{cor:OrthSubspaceOK} requires. {Moreover, for the computed basis ${\mbf U}_c$ there are likely some $\sigma \in \bb C$ for which the strict condition \cref{eq:ExistCondMatVecExact} will fail (even in exact arithmetic), but for relatively low tolerances $\tau_2$, the condition \cref{eq:ExistCondMatVec} will pass.}  Thus, in either case, we expect the computed $M_0$ via the least squares problem will be an approximation to $H(\sigma)$. Therefore, in the next section, we provide a method to estimate this error in $M_0$ recovered via \cref{eq:LSSolution} without requiring the knowledge of the true value $H(\sigma)$.

\subsection{Windowing and error estimation}
\label{sec:windowing}

As we already discussed and motivated above, in practice we expect $M_0$ to only approximate $H(\sigma)$. Therefore, given the input/output data pair $(\bb U, \bb Y)$, obtaining multiple estimates of $M_0$, which then could be used for error estimation will be crucial. We will achieve this goal via a windowing strategy.

Choose  $t < T$ and break the data into 
$T-t+1$ windows of length $t+1$:
{
\begin{equation}
    \bb U_k = \left[ \begin{array}{ccc} u[k] & \ldots & u[{k+t}]\end{array}\right]  \in \bb R^{t+1},  \quad \bb Y_k = \left[ \begin{array}{ccc}y[k] & \ldots & y[{k+t}]\end{array}\right] \in \bb R^{t+1}
\end{equation}
for $k = 0,1,\ldots,T-t$.}  We  now define
\begin{equation}
    \label{eq:GnkDef}
    \mbf G_{k,n} \coloneqq \begin{bmatrix}\bb H_{n}(\bb U_k)\\\bb H_{n}(\bb Y_k)\end{bmatrix} \in \mathbb C^{2(n+1) \times (t-n+1)}
\end{equation}
to be the matrix \cref{eq:GnDef}, but constructed using only the data in $\bb U_k$ and $\bb Y_k$, i.e., from the $k$-th window.  For a given $\sigma \in \bb C$, this will allow us to calculate multiple estimates $M_{0,k}$ for $H(\sigma)$, which will then provide us with an estimation for the error in the recovery of $M_0$ as we explain next.

We first choose a subset consisting of $K \ll T-t+1$ data windows of length $t+1$:
\begin{equation}
    \label{eq:WindowSubsetIntroduce}
    I_K = \{(\bb U_{k_1},\bb Y_{k_1}),\ldots,(\bb U_{k_K},\bb Y_{k_K})\} \subset \{(\bb U_{1},\bb Y_{1}),\ldots,(\bb U_{T-t+1},\bb Y_{T-t+1})\}.
\end{equation}
We choose a subset of all available $T-t+1$ windows because solving the least squares problems as in \cref{eq:LSSolution} for the full set of windows could be costly and (as will be shown later) is not needed.  Guidance on how to choose $K$ will be given later in the section. However, if desired, one can simply choose $K=T-t+1$.

Now, for each of the $K$ windows, we first find an orthonormal basis 
$\mbf U_{k_i} \in \bb R^{(2n+1)\times p_i}$
for the range of $\mbf G_{k_i,n}$ and then for each $k_i,i=1,2,\ldots,K$, {check the (rank) conditions \cref{eq:ExistCondMatVec} and \cref{eq:UniqueCondMatVec}.  Then assuming the these conditions hold, we }solve the least-squares problem
\begin{equation}
    \label{eq:LSSolution_window}
    \argmin_{\mbf x \in \bb R^{p_i+1}}
    \left\|\begin{bmatrix}
        \mbf U_{k_i}
        & \mbf z(\sigma)
    \end{bmatrix}
    \mbf x
    -
    \mbf b(\sigma)
    \right\| \quad \mbox{and~set~}\quad M_{0,k_i} \coloneqq \mbf x_{p_i+1}
\end{equation}
to obtain 
$\{M_{0,k_i}\}_{i=1}^{K}$, a total of $K$ estimates of $H(\sigma)$. As mentioned in \cite{burohmanBSC2020informativity}, the initial condition of the state  does not effect the recovered transfer function value.  This is crucial, since for each window except possibly for $k=0$, the input-output data $(\bb U_k, \bb Y_k)$ was obtained from the system with a non-zero initial state.

For the best accuracy, a good heuristic is to keep only the values of $M_{0,k_i}$ that correspond to the smallest residuals of the least-squares problem~\cref{eq:LSSolution_window} for $k_i=1,2,\ldots,K$. In our implementation, we achieve this by choosing $I_W$, a subset of $I_K$ consisting of $W \leq K$ windows:
$$I_W = \{(\bb U_{\ell_1},\bb Y_{\ell_1}),(\bb U_{\ell_2},\bb Y_{\ell_2}),\ldots,(\bb U_{\ell_W},\bb Y_{\ell_W})\} \subseteq I_K,$$
where the elements of $I_W$ are chosen as the windows that have lowest least-squares residuals in~\cref{eq:LSSolution_window}.
We then define our final estimate $M_0$ to $H(\sigma)$ to be the mean of our $W$ estimates, i.e.,
\begin{equation}
    \label{eq:M0isAverage}
    M_0 \coloneqq \frac{1}{W}\sum_{i = 1}^{W} M_{0,\ell_i}.
\end{equation}
Given the estimates $\{M_{0,\ell_i}\}_{i=1}^{W}$, let $s$ denote their sample standard deviation and normalized sample standard deviation $s_W$, namely
\begin{center}
    \begin{subequations}
    \parbox[l]{\textwidth*16/30}{
        \begin{equation}
            s = \sqrt{\frac{\sum_{i=1}^{W}\left|M_{0,\ell_i}-M_0\right|^2}{W-1}},~
        \end{equation}
    }
    \parbox[l]{\textwidth*12/30}{
        \begin{equation}
            \label{eq:sw}
            {s_W = \left|\frac{s}{{M_0}}\right|}.
        \end{equation}
    }
    \end{subequations} 
\end{center}
assuming $M_0 \neq 0$ (the $M_0=0$ case can be handled without normalization). Let 
\begin{equation}
    \label{eq:relErrInWindowing}
    \epsilon_{rel} = \left|\frac{H(\sigma) - M_0}{H(\sigma)}\right |
\end{equation}
denote the relative error in the recovered value of $H(\sigma)$.
In our numerical examples, we have found  $s_W$ in~\cref{eq:sw} to be a consistently good indicator for 
the true relative error $\epsilon_{rel}$ in~\cref{eq:relErrInWindowing}.  Before we provide a justification for this claim, we illustrate it via a numerical example.

We consider a system $\mathcal S_1$ as in \cref{eq:LinDiscSys} with transfer function $H_1(z)$.  The system $\mc S_1$ has 100 random poles in the unit disc and 100 random residues. We investigate recovering 
$H_1(e^{\mbf i\omega_i}),\, i = 1,2,\ldots,100$
using \cref{eq:M0isAverage}
where $\omega_i$ are logarithmically spaced in $[10^{-3},\pi)$.  The data was collected by simulating $S_1$ for $T = 1,\!000$ time steps with a Gaussian random input $\bb U \in \bb R^{T+1}$ to obtain the corresponding output $\bb Y \in \bb R^{T+1}$.  We calculate $M_{0,k_i}$ via \cref{eq:LSSolution_window} for $K = 20$ windows of length $t = 3n$ (out of $T-t+1 = 701$ possible windows) and take only the $W = 10$ estimates with lowest residual to use in calculating $M_0$ via \cref{eq:M0isAverage}.

\begin{figure}[!htb]
    \centering
    \includegraphics[width = .8\textwidth]{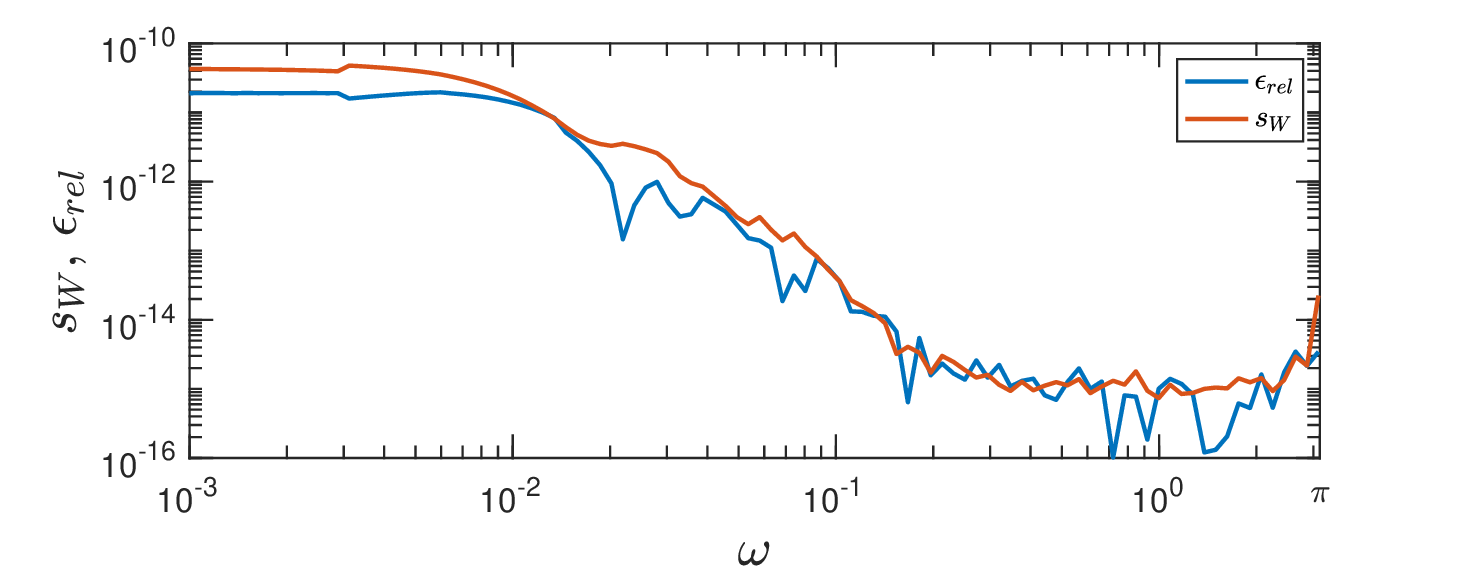}
    \caption{The normalized standard deviation ($s_W$) of $\{M_{0,\ell_i}(e^{\mbf i\omega})\}_{i=1}^{W}$ provides a good indicator for the relative error ($\epsilon_{rel})$ of $M_0(e^{\mbf i\omega})$ to $H_1(e^{\mbf i\omega})$.}
    \label{fig:NstdShowsErr}
\end{figure}

In \cref{fig:NstdShowsErr}, we plot the normalized sample standard deviation $s_W$ and the relative error $\epsilon_{rel}$ in recovery of approximations $M_0$ to $H_1(e^{\mbf i \omega})$.
The figure shows that $s_W$   typically well approximates the relative error $\epsilon_{rel}$ in the recovered transfer function value \cref{eq:relErrInWindowing}. 
We have observed this behavior consistently throughout our numerical examples.
This becomes an important tool since we now have an approximation of the relative error in recovered transfer function value $M_0(\sigma)$ that does not require the knowledge of $H(\sigma)$.  We also see that using a small fraction of all available windows does a good job of producing quality approximations to $H(\sigma)$.

We examine the dependence of the accuracy of the recovered frequency information and error indicator on the number of windows $K$ by finding approximations $M_0(\sigma)$ to $H_1(\sigma)$ 
with various number of windows $K$ for the same $\sigma \in \bb C$ using the same time-domain input-output data $\bb U \in \bb R^{T+1}$ and $\bb Y \in \bb R^{T+1}$ with $T = 1,\!000$.  We keep the size of the subset of windows constant at $W = 10$ while varying the total number of windows $K$.

\begin{figure}[!htb]
    \centering
    \includegraphics[trim={0 2cm 0 2cm},clip,width = .75\textwidth]{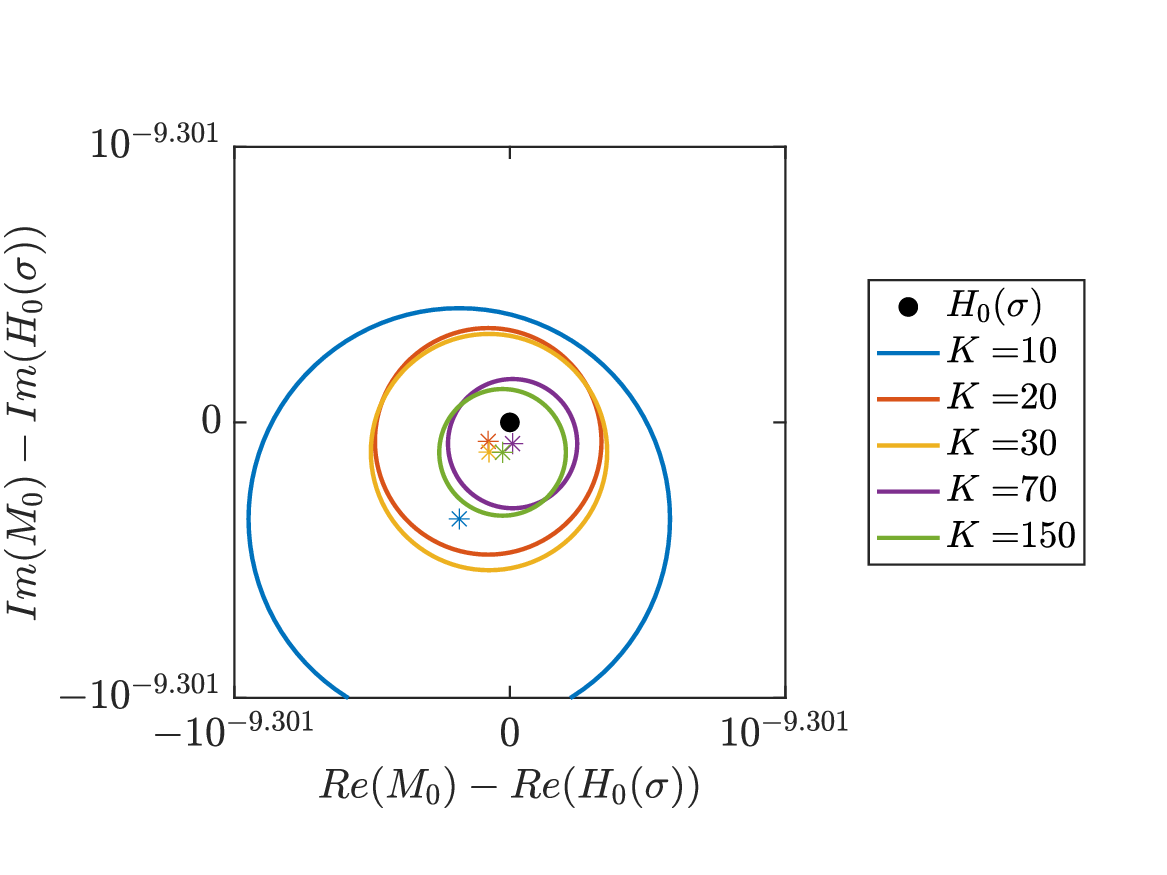}
    \caption{Error in $M_0(\sigma)$ (stars) and the boundary of the set of points one standard deviation from $M_0(\sigma)-H_1(\sigma)$ (solid circles) for different values of $K$.}
    \label{fig:OnlyNeedNW=10}
\end{figure}

\cref{fig:OnlyNeedNW=10} shows the results of recovering $H_1(\sigma)$ for $\sigma = e^{10^{-2} \mbf i}$ using the windowing framework via \cref{eq:LSSolution_window} and \cref{eq:M0isAverage} for $K = 10,20,30,70,150$, where the $W=10$ values used in \cref{eq:M0isAverage} are chosen as the values of $M_{0,\ell_i}$ corresponding to the lowest $W$ residuals of \cref{eq:LSSolution_window}.
In this figure, for each value of $K$, we  plot the absolute error in the recovered value of $M_0$ (the origin corresponds to exact recovery) and the boundary of the set $\{z \in \bb C \,|\, |z-M_0|\leq s_W$\}. 
We observe that the accuracy of the recovered transfer function value does not depend strongly on $K$.
For values of $K > W=10$, there is little difference in  the quality of the recovered $M_0$.  Thus, there is a benefit to taking a subset of all recovered $\{M_{0,k_i}\}_{i=1}^K$ values to use in calculation of $M_0$ via \cref{eq:M0isAverage} (i.e., we should have $K > W$), but it does not appear that there is a significant benefit to using large $K$.
Based on our numerical experiments (not just for this example but using all the tests we have run), we recommend choosing  $K \approx 20$ and $W \approx 10$ as a starting point, to be increased if needed. 

Finally, we observe that the set of points within one standard deviation of $M_{0}$ contains $H_1(\sigma)$ for each value of $K$. 
If the values of $M_{0,k_i}$ were to sampled from a normal distribution centered at $H(\sigma)$
(we cannot claim this is the case here),  approximately $68\%$ of the estimates would be expected to be within one standard deviation of $H(\sigma)$.
These observations at least help explain why our standard deviation error indicator $s_W$ in~\cref{eq:sw} does a good job of approximating the true relative error $\epsilon_{rel}$ 
in~\cref{eq:relErrInWindowing}.  If for a given set of approximants $\{M_{0,\ell_i}\}_{i=1}^{W}$ the true value of $H(\sigma)$ is within $s$ of $M_0$, then we have
$|H(\sigma)-M_0(\sigma)| < s.$
Thus, we obtain
$$ \epsilon_{rel} = \left |\frac{H(\sigma)-M_0(\sigma)}{H(\sigma)}\right | < \left|\frac{ s}{H(\sigma)}\right| \approx \left |\frac{ s}{M_0}\right| = s_W.$$
We emphasize that we are \emph{not} claiming $s_W$ to be a true upper bound for $\epsilon_{rel}$.
While this usually holds numerically, one can see in \cref{fig:NstdShowsErr} that for some values of $\omega$ the relative error is greater than the normalized standard deviation.
However, overall, $s_W$ proves to be a good error indicator.

\subsection{Flexibility in choosing system order}
\label{sec:calc_n_hat}

Recall that the theory presented in \cref{sec:DataInform} assumed that we knew the system order, $n$.  
This is an issue in practice if we wish to recover frequency information $H(\sigma)$ from only time-domain data $(\bb U,\bb Y)$ via \eqref{eq:calcM0Orig} without knowledge of the system. 
In \cref{prop:DontNeedFulln} below we show that one does not  need an exact knowledge of $n$ and can use any $\hat n \geq n$.

\begin{lemma}
    \label{prop:DontNeedFulln}
    Let $\mathcal S$ be an 
    order $n$ linear discrete-time SISO system with transfer function $H(z)$.
    Let $\hat n \geq n$ and $\sigma \in \bb C$ not a pole of $\mathcal S$.  If there exists $M_0 \in \bb C$ such that
    \begin{equation} \label{eq:nhatinform}\Sigma^{\hat n}_{\bb U,\bb Y} \subseteq \Sigma^{\hat n,0}_{\sigma,M_0},\end{equation}
    i.e., if the data is informative for interpolation at $\sigma$ using $\hat n \geq n$ in place of $n$, 
    then $M_0 = H(\sigma)$.
\end{lemma}

\begin{proof}
    Recall from \eqref{eq:TransFunIsRatioanlExpanded} that $H(z)= \frac{Q(z)}{P(z)}$ is an order $n$ rational function in $z$.
    Choose $\hat n \geq n$ and define the order $\hat n$ (non-minimal) rational function $\hat H(z)$ by choosing any $\alpha \neq \sigma \in \bb C$ with $|\alpha| < 1$ and setting
    $$\hat H(z) \coloneqq \frac{Q(z)}{P(z)}\frac{(z-\alpha)^{\hat n-n}}{(z-\alpha)^{\hat n-n}}.$$
    Note that for all $\mu \neq \alpha$, we have
    $\hat H(\mu) = H(\mu).$
    Since $\hat H(z)$ corresponds to a non-minimal order $\hat n$ system $\hat{\mathcal S}$, for any input $\bb U$, the time-domain outputs $\bb Y$ and $\hat{\bb Y}$ produced by $\mathcal S$ and $\hat{\mathcal S}$, respectively, will be equal.  So
    $\hat{\mathcal S} \in \Sigma^{\hat n}_{(\bb U,\bb Y)}.$
    Then since $\hat H(\sigma) = H(\sigma)$, if there exists $M_0 \in \bb C$ such that~\eqref{eq:nhatinform} holds,
    it must be that $M_0 = H(\sigma),$ since at least one member of $\Sigma^{\hat n}_{(\bb U,\bb Y)}$ (namely $\hat{\mathcal S}$) has transfer function value $H(\sigma)$ at $\sigma$.
\end{proof}

\begin{remark}
    In \cite{burohmanBSC2020informativity}, the authors hint at such a result by not assuming minimality of the underlying system $\mc S$.  However, the  result is formulated using only {one specific} $n$, the order of the (potentially non-minimal) system.  The contribution of \cref{prop:DontNeedFulln} is to show that knowledge of any $\hat n \geq n$ is sufficient to calculate frequency information from time-domain data, thus showing that there are infinitely many values of `$n$' that will allow for accurate recovery of frequency information. We also emphasize that in light of \cref{prop:DontNeedFulln}, there is no need to assume minimality of $\mc S$ (as we did at the beginning of \cref{subsec:ProbForm} to simplify the discussion).
\end{remark}

In our work we assume access only to the time-domain data $(\bb U,\bb Y)$, not to the knowledge of $n$ and hence we do not know an $\hat n \geq n$.  However, methods such as MIMO Output-Error State Space (MOESP) \cite{Li2012MatrixPencils,VerhaegenD1991Subspace1} are able to provide an estimate $N \approx n$.  We then use $N$ in place of $n$ to construct the matrix $\mbf G_{N}$ as in \eqref{eq:GnDef} and the vectors $\mbf z(\sigma)$ as in \eqref{eq:zSigmaDef} and $\mbf b(\sigma)$ as in \eqref{eq:bSigmaDef}; and proceed according to the results in \cref{sec:Implementation}.  If it happens that $N \ll n$,  it is possible that $M_0$ is recovered with large error.
The error indicator described in \cref{sec:windowing} will allow us to quantify the error incurred {for this substitution}.  Since \cref{prop:DontNeedFulln} guarantees that there is no penalty for \emph{overestimating} $n$, we are able to try the same procedure with $\tilde n > N$ until the error indicator is acceptably low.  In the numerical example below, we demonstrate \cref{prop:DontNeedFulln} and show that a slight underestimate of $n$ typically works well in practice.

We conduct the following experiment: for the system $\mathcal S_0$ with transfer function $H_0(z)$ (introduced in \cref{sec:orthogonalization}), we simulate $\mathcal S_0$ for $T=1,\!000$ time steps to obtain $\bb U \in \bb R^{T+1}$ and $\bb Y \in \bb R^{T+1}$.  Then, following \cref{sec:windowing}, using only $\bb U$ and $\bb Y$ for each
$\tilde n = 20,21, \ldots 200$
we form $K = 20$ data windows of length $3\tilde n$, construct orthogonal bases $\mbf U_{k_i} = \orth (\mbf G_{k_i,\tilde n})$, and solve \eqref{eq:LSSolution_window} for each of the $K$ orthogonal bases to obtain $\{M_{0,k_i}^{(\tilde n)}\}_{i=1}^K$, which provides a total of $K$ estimates of $H_0(\sigma)$ obtained by using $\tilde n$ in place of $n$.  Then as in \cref{sec:windowing} we use the $W = 10$ values of $M_{0,k_i}^{(\tilde n)}$ corresponding to smallest residuals of \eqref{eq:LSSolution_window} to calculate $M_{0}^{(\tilde n)} \approx H_0(\sigma)$ via \eqref{eq:M0isAverage}.

Since we have access to the true value of $H_0(\sigma)$, we calculate the relative error as
\begin{equation}
    \epsilon_{rel}(\tilde n) =\frac{|H_0(\sigma)-M_{0}^{(\tilde n)}(\sigma)|}{|H_0(\sigma)|}.
\end{equation}
We also keep track of $s_W(\tilde n)$, the standard deviation error indicator \eqref{eq:sw} for $M_0^{(\tilde n)}(\sigma)$.
The dependence of $\epsilon_{rel}(\tilde n)$ and $s_W(\tilde n)$ on $\tilde n$ is displayed in \cref{fig:EpsRelVsNTilde}. First observe that, as before,  $s_W$ closely tracks the true error $\epsilon_{rel}(\tilde n)$.

\begin{figure}[!htb]
	\centering
    \includegraphics[width=.8\textwidth]{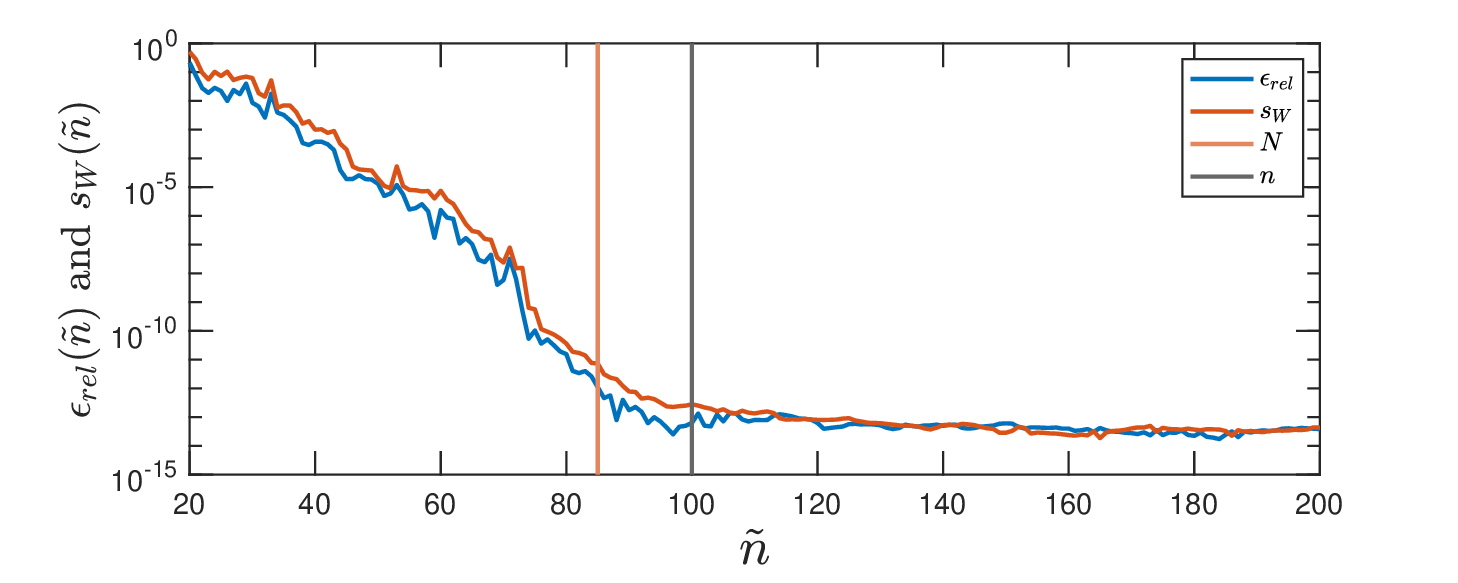}
	\caption{Relative error in recovered frequency information for $\tilde n$ between $20$ and $200$.  The vertical orange line indicates the $N \approx n$ calculated using MOESP.}
	\label{fig:EpsRelVsNTilde}
\end{figure}

\cref{fig:EpsRelVsNTilde} also provides the numerical confirmation of \cref{prop:DontNeedFulln}.  We see that for $\tilde n \geq n = 100$, the relative error in recovered frequency information $M_0^{(\tilde n)}$ is on the order of machine precision.  \cref{fig:EpsRelVsNTilde} also shows that the relative error decreases gradually as $\tilde n$ approaches $n$; there is no sharp drop off in the error.  This implies that if some error in the recovered frequency information is acceptable, one may use a value of $\tilde n < n$.  Indeed, we see the numerically estimated value $N = 85$ yields a
relative error less than $10^{-12}$.  We re-emphasize that for each value of $\tilde n$, the standard deviation error indicator $s_W(\tilde n)$ closely follows the relative error $\epsilon_{rel}(\tilde n)$.  So in practice when only $\bb U$ and $\bb Y$ are available we first numerically calculate $N$, then attempt to recover $M_0$ via \eqref{eq:LSSolution_window}, recording $s_W$.  If $s_W$ is large, we choose $\tilde n > N$ and again attempt to recover frequency information, repeating until $s_W$ is sufficiently small.

\subsection{Algorithm}
\label{Sec:Algorithm}
This section provides a summary (\cref{alg:DataInform}) of the contributions of \cref{sec:Implementation} and provides guidance on parameter choices.  References to where each line is discussed are provided.  

\begin{algorithm}[!htb]
\caption{Inferring Transfer function values  from time-domain data}\label{alg:DataInform}
\begin{algorithmic}[1]
\Require Time-domain input-output data $(\bb U,\bb Y)$ and $\{\sigma_j\}_{j=1}^m \subset \bb C$, points to learn the transfer function.\\
\Return $\{M_0(\sigma_j)\}_{j=1}^m$, estimates to $H(z)$ at each $\sigma_j$
\State Choose approximate order $\tilde n$ (\cref{sec:calc_n_hat}).
\State Set $t \gets k < T$ (we use $k = 3\tilde n$) (\cref{sec:windowing}).
\State Choose $K$ and $W$, the size of the subsets $I_K$ and $I_W$.
(\cref{sec:windowing}).
\State Choose tolerances $\tau_1$ and $\tau_2$ for rank conditions \cref{eq:UniqueCondMatVec}  and \cref{eq:ExistCondMatVec} (\cref{sec:RankCondsAsMatVecs})
\State Calculate orthogonal bases $\mbf U_{k}\in \bb C^{2(\tilde n+1)\times p}$ for $k = 1,2,\ldots,K$ data windows of length $t$ (\cref{sec:windowing} and \cref{sec:orthogonalization}).
\For{$j = 1,2,\ldots,m$}
\State Form $\mbf z(\sigma)$ and $\mbf b(\sigma)$ as in \cref{eq:zSigmaDef} and \cref{eq:bSigmaDef}.
\For{$k = 1,2,\ldots K$}
\If{\cref{eq:UniqueCondMatVec} and \cref{eq:ExistCondMatVec} are satisfied}
\begin{equation}
    \label{eq:LSProblemInAlg}
    \tag{47}
     \mbf x \coloneqq \argmin_{\hat{\mbf x} \in \bb C^{p+1}}\left\|\begin{bmatrix}\mbf U_k & \mbf z(\sigma)\end{bmatrix}\hat{\mbf x} - \mbf b(\sigma)\right\|
     ~\mbox{and~set}~M_{0,k}(\sigma_i) = {\mbf x}_{p+1}
\end{equation}
\EndIf
\EndFor
\State Select $\{M_{0,k_i}\}_{i = 1}^W \subseteq \{M_{0,k}\}_{k = 1}^K$ according to the lowest residuals of \cref{eq:LSProblemInAlg} 
\State Set $M_0(\sigma_j) \gets \frac{1}{W}\sum_{i = 1}^{W}M_{0,k_i}$  \cref{eq:M0isAverage}
\EndFor
\end{algorithmic}
\end{algorithm}

\begin{remark}
    \label{rmk:DerivativeAlg}
    \cref{alg:DataInform} is presented for the recovery of values of the transfer function $H(\sigma)$ given time domain data $(\bb U,\bb Y)$. However, the procedure to recover transfer function derivatives is nearly identical.  The only change is that the vector $\mbf b(\sigma)$ in \cref{eq:LSSolution_window} becomes $\bmat{\gamma_n^{(1)}(\sigma)^{\top} & M_0\gamma_n^{(1)}(\sigma)^{\top}}^{\top}$ (see \cref{eq:calcMkOrig}) and the condition \cref{eq:ExistCondMatVec} must be checked with the new $\mbf b(\sigma)$.
\end{remark}

\section{Numerical results}
\label{sec:Results}

In this section, we apply \cref{alg:DataInform} 
to three dynamical systems to construct DDROMs via frequency-based techniques using frequency data recovered from time-domain data using \cref{alg:DataInform}.
{These numerical experiments were performed on a 2023 MacBook Pro equipped with 16 GB RAM and an Apple M2 Pro chip running macOS Ventura 13.4.1.  All algorithms are implemented in \MATLAB version 23.2.0.2365128 (R2023b).  All code and data are available in} \cite{supAck23}.
We  examine three modekls: (i) a synthetic system with randomly placed poles in the unit disc in \cref{subsec:SynthEx}, (ii) the fully discretized heat model  from~\cite{Chahlaoui2002Benchmarks} in \cref{sec:HeatModel}, and (iii) Penzl's model~\cite{Penzl2006Algorithms} in \cref{Sec:PenzlEx}.  {For each system we calculate transfer function values and derivatives from time-domain data at specified points $\{\sigma_i\}_{i=1}^m \subset \bb C$ via \cref{alg:DataInform} and record the relative errors as
\begin{equation}
    \label{eq:ErrRecoverdData}
    \varepsilon_0 = \dfrac{\|\mbf H_0 - \widehat{\mbf H}_0\|_2}{\|\mbf H_0\|_2}\quad \mbox{and} \quad
        \varepsilon_1 =
    \dfrac{\|\mbf H_1 - \widehat{\mbf H}_1\|_2}{\|\mbf H_1\|_2},
\end{equation}
where $\mbf H_0 = [H(\sigma_1),\ldots,H(\sigma_{m})]^{\top}$ is the vector of true transfer function values, $\mbf H_1 = [H'(\sigma_1),\ldots,H'(\sigma_{m})]^{\top}$ is the vector of true  derivative values, and $\widehat{\mbf H}_0 = [M_0(\sigma_1),\ldots,M_0(\sigma_{m})]^{\top}$ and $\widehat{\mbf H}_1 = [M_1(\sigma_1),\ldots,M_1(\sigma_{m})]^{\top}$ are the vectors of recovered transfer function values and derivatives, respectively.}   In every example, we use $\tau_1 = \tau_2 = 10^{-10}$ (\cref{sec:RankCondsAsMatVecs}).  We also keep the number of windows $K = 20$ and the size of the subset of these windows used to calculate $M_0$, $W=10$ constant for each example (\cref{sec:windowing}), except for Penzl's example where we use $K = 40$.

We then use the recovered frequency information to construct DDROMs using some of the well-established frequency domain reduced-order modeling techniques. Specifically, we use the Loewner framework~\cite{MayoA2007LoewnerFramework}, the Hermite Loewner framework~\cite{MayoA2007LoewnerFramework}, and Vector Fitting \cite{DeschrijverMDZ2008FastVF,Gustavsen2006PoleRelocVF,semlyenG1999VF}. 
Our focus in these numerical examples is to employ data-informativity in frequency-based techniques and investigate the impact of inferred values on the DDROM construction. In a future benchmark paper, we will compare and study a wider class of methods including those that produce DDROMs directly from time-domain data.

To form an order-$r$ DDROM of a transfer function $H(z)$, the Loewner framework requires two sets of interpolation points: $\{\sigma_i\}_{i=1}^r$, and $\{\mu_i\}_{i=1}^r$ and their corresponding values $\{H(\sigma_i)\}_{i=1}^r$, and $\{H(\mu_i)\}_{i=1}^r$.  The resulting DDROM, $H_r^{L}(z)$, interpolates $H(z)$ at each interpolation point, i.e.,
$H_r^{L}(\sigma_i) = H(\sigma_i)$ and $H_r^{L}(\mu_i) = H(\mu_i).$

The Hermite Loewner framework takes 
$\sigma_i = \mu_i$, i.e., uses only one set of interpolation points $\{\sigma_i\}_{i=1}^r$, but requires knowledge of transfer function values $\{H(\sigma_i)\}_{i=1}^r$ and derivatives $\{H'(\sigma_i)\}_{i=1}^r$.  Then the resulting DDROM, $H_r^{HL}(z)$, interpolates $H(z)$ and $H'(z)$ at each interpolation point, i.e.,
$H_r^{L}(\sigma_i) = H(\sigma_i)$ and ${H'}_r^{L}(\sigma_i) = H'(\sigma_i).$
For details on the Loewner framework, see e.g. \cite{AntoulasBG2020Book,BennerCOW2017ModRedApproxTA,MayoA2007LoewnerFramework}.

Finally, given $m > r$ points $\{\sigma_i\}_{i=1}^m \subset \bb C$ and the transfer function samples $\{H(\sigma_i)\}_{i=1}^m$, Vector Fitting performs a least squares fit to the given data.  Specifically, Vector Fitting tries to find an  order $r$ rational function $H_r^{VF}(z)$ that minimizes the least-squares error
{$\sum_{i = 1}^m\left|H(\sigma_i) - H_r(\sigma_i)\right|^2$.}
For details of Vector Fitting, we refer the reader to \cite{DeschrijverMDZ2008FastVF,DramacGB2015QuadVF,Gustavsen2006PoleRelocVF,semlyenG1999VF}.

To measure the performance of our DDROMs we use the $\mc H_{\infty}$ norm of a transfer function $H$ defined as $\|H\|_{\mc H_{\infty}} = \max_{\omega \in [-\pi,\pi]}|H(e^{\mbf i\omega})|$.

\subsection{Synthetic example}
\label{subsec:SynthEx}
We construct an order $n=1000$ system $\mathcal S_2$ with transfer function
$H_2(z)$ with $1000$ randomly placed poles in the open unit disc and with randomly generated residues, which are both closed under conjugation.  We simulate $\mc S_2$ for $T = 1000$ time steps with random Gaussian input.  We then use \cref{alg:DataInform} (and its derivative version, see \cref{rmk:DerivativeAlg})
to  compute the estimates $M_0(\sigma_i)$ (for the transfer function value $H_2(\sigma_i)$) and  $M_1(\sigma_i)$ (for the derivative value $H_2'(\sigma_i)$) where
$\sigma_i = e^{\mbf i\omega_i}$ with $\omega_i$  logarithmically spaced in $[10^{-2},\pi)$ for
$i = 1,2,\ldots, 400$. Using~\cite{VerhaegenD1991Subspace1}, we estimate (and take) {$N = 183$}, the numerically estimated system order.

\begin{figure}
     \centering
     \subfloat{\includegraphics[scale = .32]{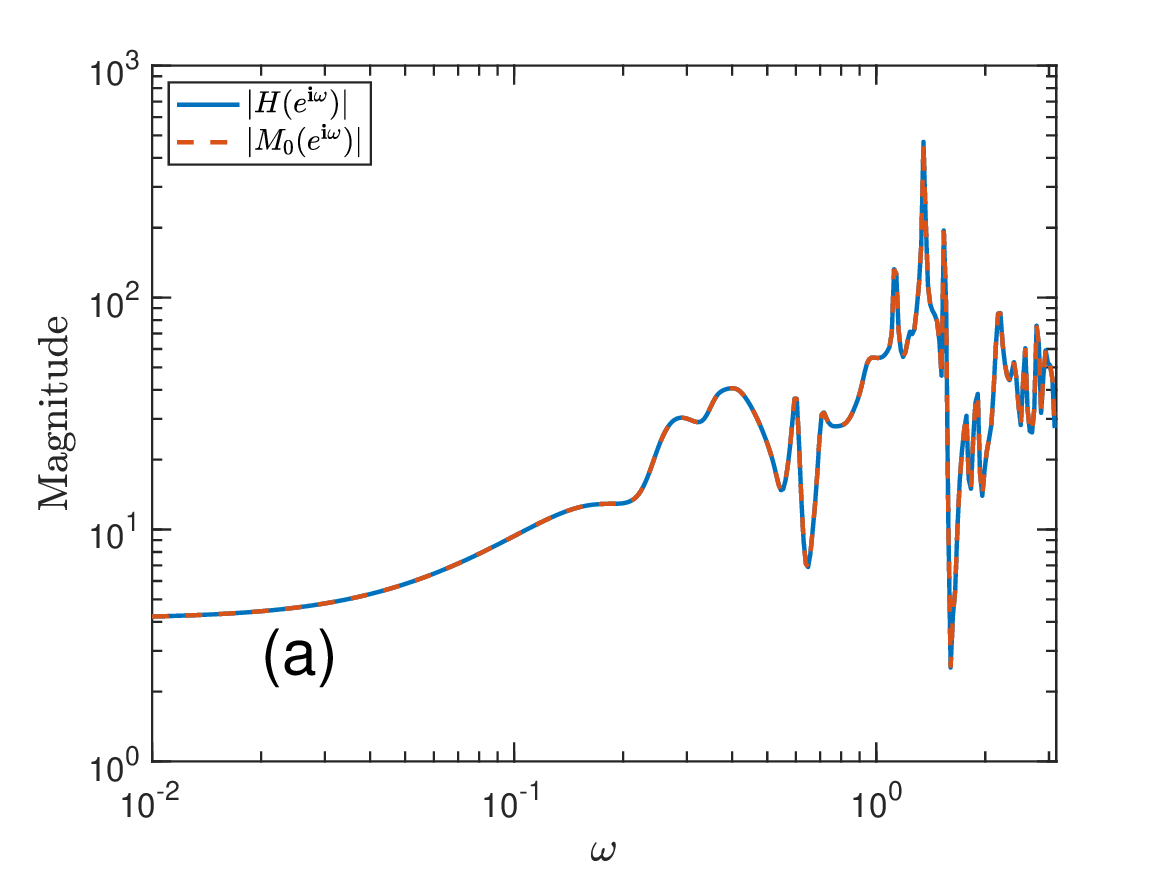}}
     \subfloat{\includegraphics[scale = .32]{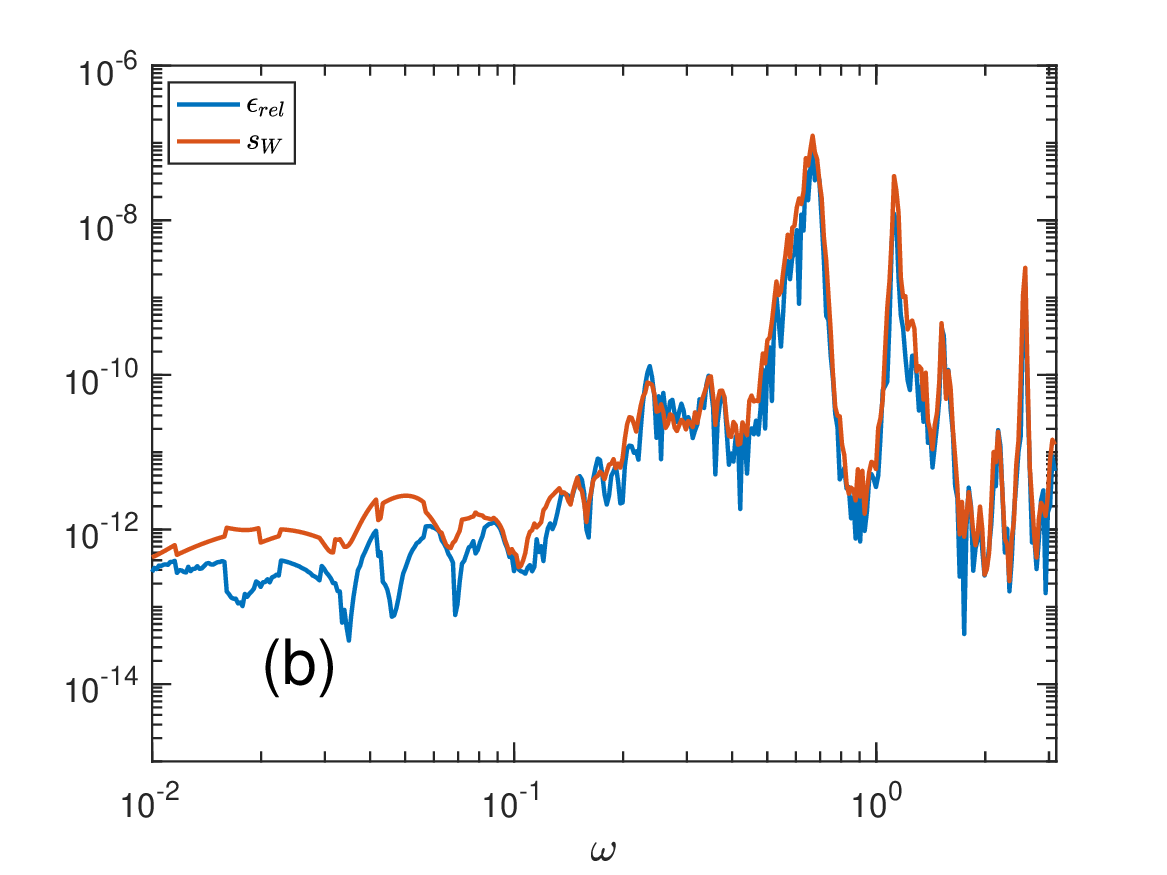}}
     
     \subfloat{\includegraphics[scale = .32]{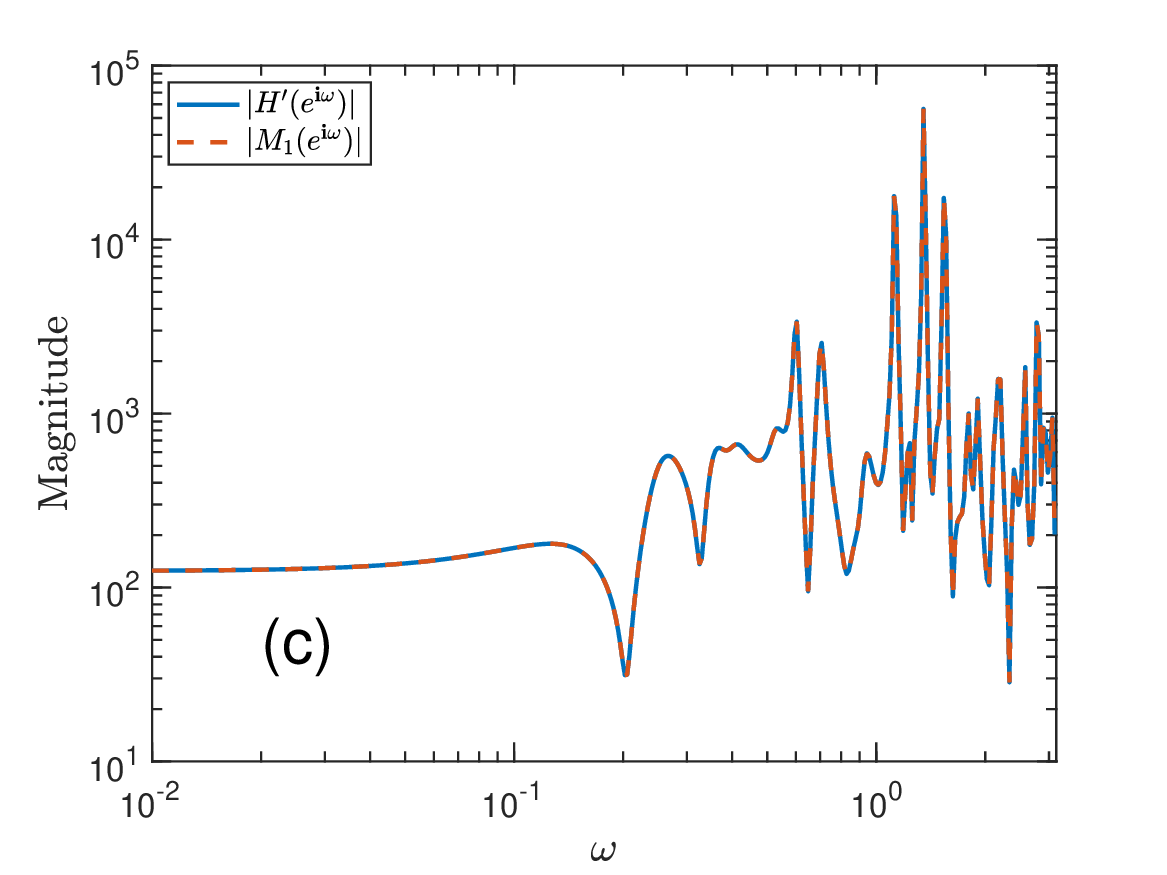}}
     \subfloat{\includegraphics[scale = .32]{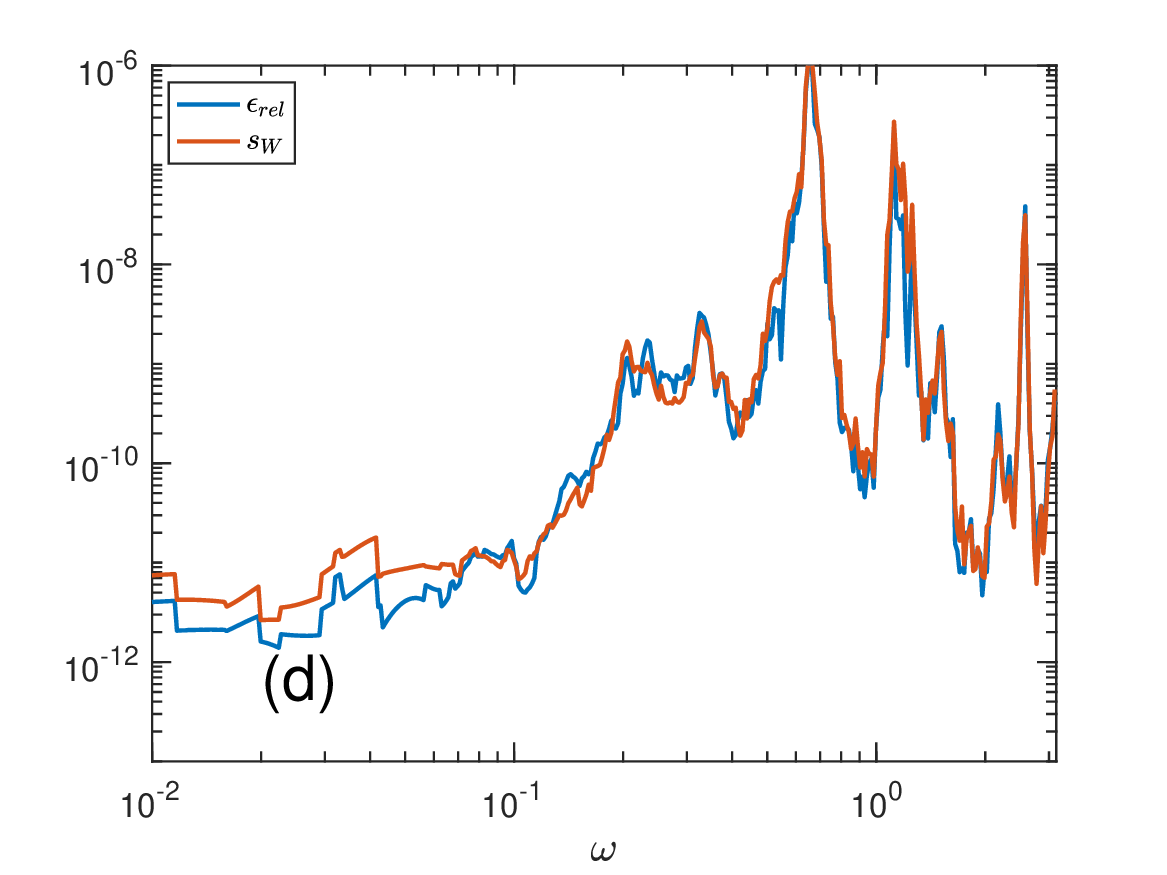}}
     \caption{Frequency responses ((a) and (c)) and point-wise relative errors (plots (b) and (d)) of DDROMs approximating $H_2$, constructed using frequency data recovered from time-domain data via \cref{alg:DataInform} ((a) and (b)) and true frequency data ((c) and (d)).}
     \label{fig:RelErrAndFreqRespSynth}
\end{figure}

{We show the accuracy of recovered transfer function values  $M_0(e^{\mbf i \omega_i}) \approx H_2(e^{\mbf i \omega_i}) $ and recovered derivative values $M_1(e^{\mbf i \omega_i}) \approx H_2'(e^{\mbf i \omega_i}) $, respectively, in the top and bottom rows of 
\cref{fig:RelErrAndFreqRespSynth} together with the point-wise relative error $\epsilon_{rel}$ and 
the normalized standard deviation $s_W$ in both cases. As the figure  illustrates, both 
$H_2(e^{\mbf i \omega_i})$  and $H_2'(e^{\mbf i \omega_i})$
are recovered to high accuracy, with a maximum point-wise relative error of $8.27 \times 10^{-8}$ in the transfer function values and a maximum point-wise relative error of  $1.23 \times 10^{-6}$ in the derivatives. 
We also point out that the normalized standard deviation, $s_W$ (computed purely from the data without any knowledge of true values) provides an accurate estimate of the true relative error $\epsilon_{rel}$.} 
Using \cref{eq:ErrRecoverdData}, we calculate the vectorized (as opposed to pointwise) relative errors in~\cref{eq:ErrRecoverdData} to obtain $\varepsilon_0 = 3.10 \times 10^{-9}$ and $\varepsilon_1 = 6.04 \times 10^{-8}$.
Note that the derivative information is recovered with a lower accuracy.  This is expected, as the method to recover $M_1$ uses $M_0$ (see \cref{eq:calcMkOrig}).  So any inaccuracies in recovering $M_0$ will compound in recovering $M_1$.

We now use the recovered values and derivatives of $H_2(z)$ to form order $r = 100$ DDROMs of $\mathcal S_2$.  Since $\mathcal S_2$ is a real system, we obtain conjugate information on the unit circle at no additional cost using  $H_2(\overline{\sigma_i}) = \overline{H_2(\sigma_i)}$ where
$\sigma_i = e^{\mbf i \omega_i}$ for $i = 1,2,\ldots, 400$.
We supply all of our recovered 800 values (and, when appropriate, its derivatives) to construct order $r=100$ DDROMs via the Loewner framework ($\widehat H_{\texttt{L}}$), the Hermite Loewner framework ($\widehat H_{\texttt{LH}}$), and Vector Fitting ($\widehat H_{\texttt{VF}}$). We also produce DDROMs using the \emph{true} frequency information via each of the three methods and denote them by $\widetilde H_{L}$ (Loewner), $\widetilde H_{LH}$ (Hermite Loewner), and $\widetilde H_{VF}$ (Vector Fitting). {We note that the Hermite Loewner and Loewner models are truncated to order $r=100$ using the singular value decomposition of the resulting Loewner matrices, see \cite{AntoulasBG2020Book} for details.} 
{When partitioning the interpolation points for the Loewner framework, we follow the guidance of \cite{EmbreeI2022} and interweave the two sets of interpolation points. {We employ these procedures for Loewner-based models for all the numerical examples.}}

In \cref{fig:RelErrAndFreqRespSynthROMs}-(a), we show the amplitude frequency response plot of the original model $H$ and together with those of the three DDROMs, $\widehat H_{\texttt{L}}$, $\widehat H_{\texttt{HL}}$, and  $\widehat H_{\texttt{VF}}$,
obtained from the recovered frequency information, illustrating that they all accurately match the original dynamics. The corresponding error plots are shown in \cref{fig:RelErrAndFreqRespSynthROMs}-(b). In \cref{fig:RelErrAndFreqRespSynthROMs}-(c) and \cref{fig:RelErrAndFreqRespSynthROMs}-(d),
we repeat the same procedure for the 
 three DDROMs, $\widetilde H_{\texttt{L}}$, $\widetilde H_{\texttt{HL}}$, and  $\widetilde H_{\texttt{VF}}$, obtained from the true frequency information. These plots illustrate that for this example the DDROMs obtained via recovered frequency data retain almost the same approximation quality as those obtained from true data. This is a  promising result since we are able to capture the accuracy of the  well-established frequency-domain techniques without having access to the frequency samples.

To better quantify the impact of learning the reduced models from recovered frequency data, we compute various relative $\mathcal H_{\infty}$ distances in 
\cref{tab:H2ErrSynth}. While the first row shows the relative $\mathcal H_\infty$ distances between the full model and the three DDROMs obtained via recovered data, the second row shows the same distances for the three DDROMs obtained via exact frequency response data. As these numbers illustrate, in this example, learning the reduced models via recovered data does not impact the $\mathcal H_\infty$ performance, (at least in the leading three significant digits). The last row of \cref{tab:H2ErrSynth} shows the relative $\mathcal H_\infty$ distance between the DDROMs, for each method, obtained via true and recovered data. These numbers more clearly illustrate that in this example the data informativity framework allows us to accurately capture the performance of the frequency-based modeling techniques without having direct access to the data. 

 \begin{figure}
     \centering
     \subfloat{\includegraphics[scale = .32]{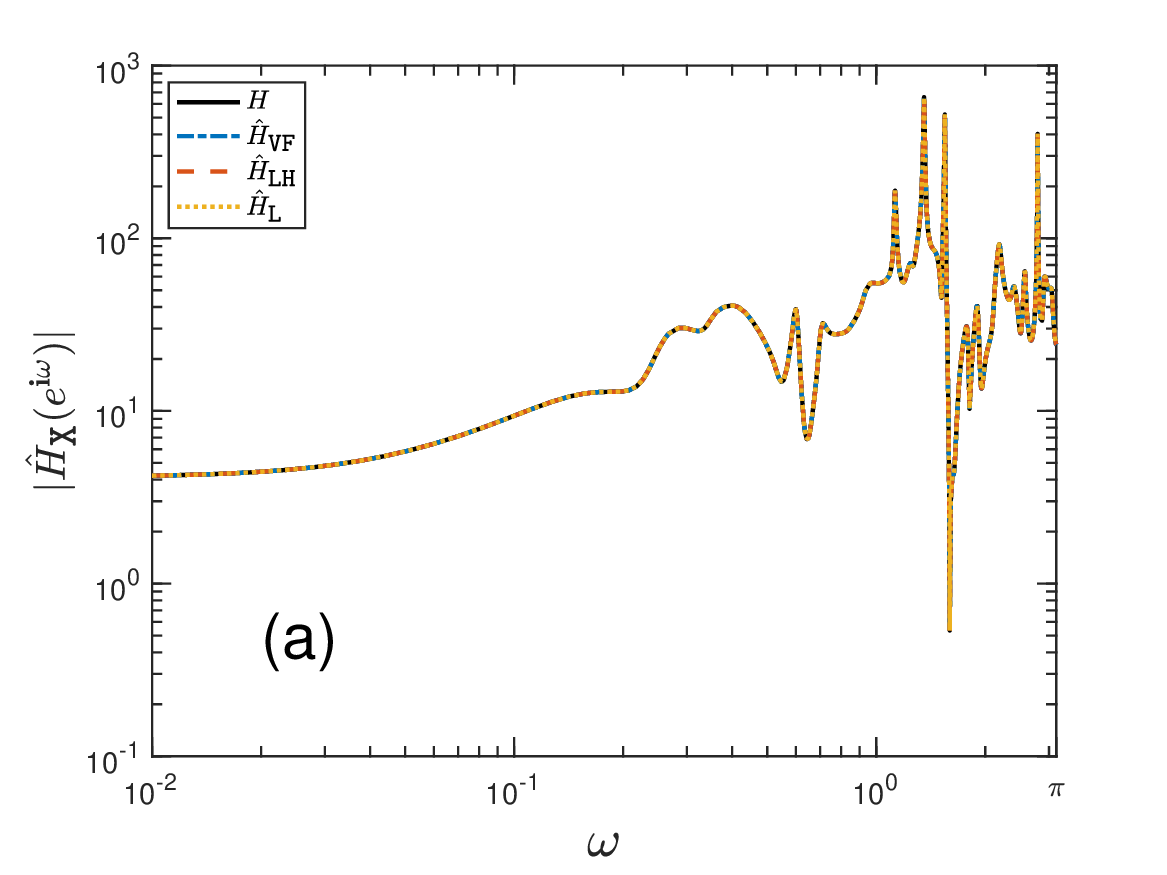}}
     \subfloat{\includegraphics[scale = .32]{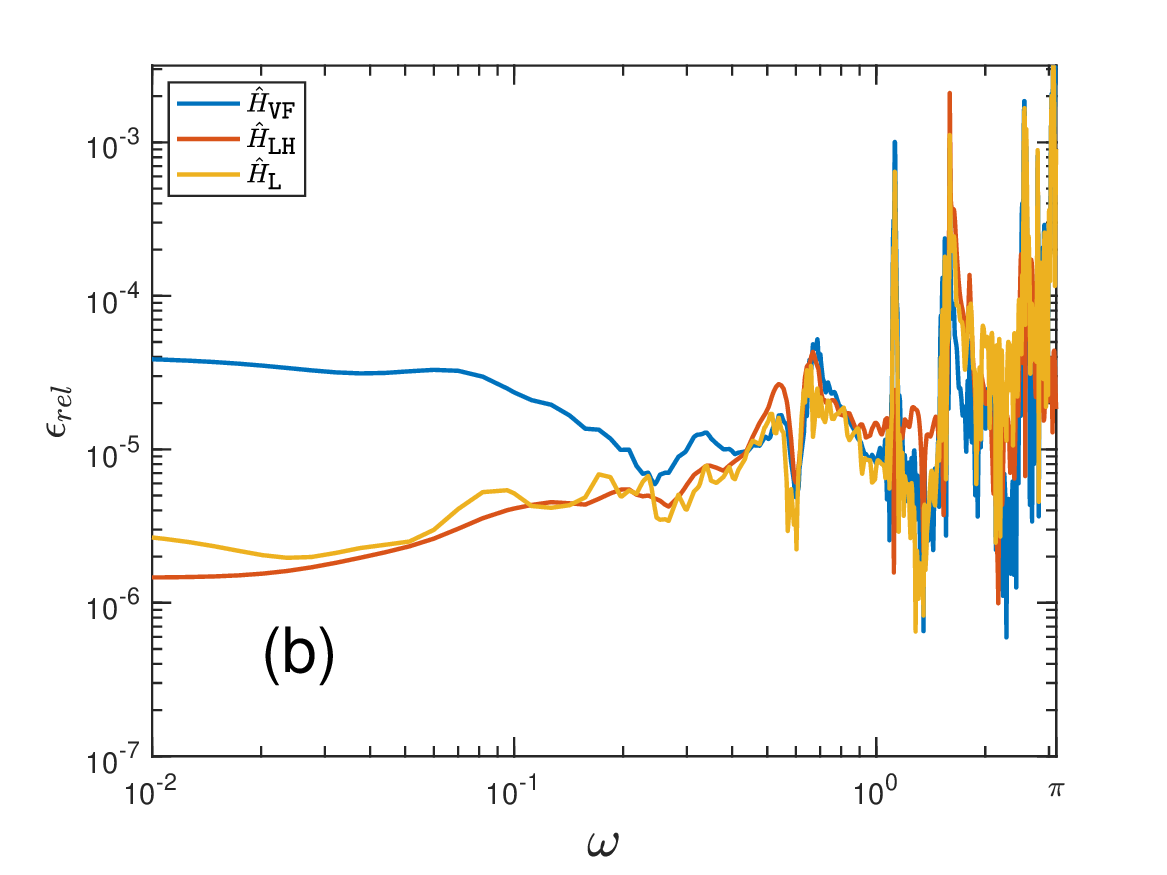}}
     
     \subfloat{\includegraphics[scale = .32]{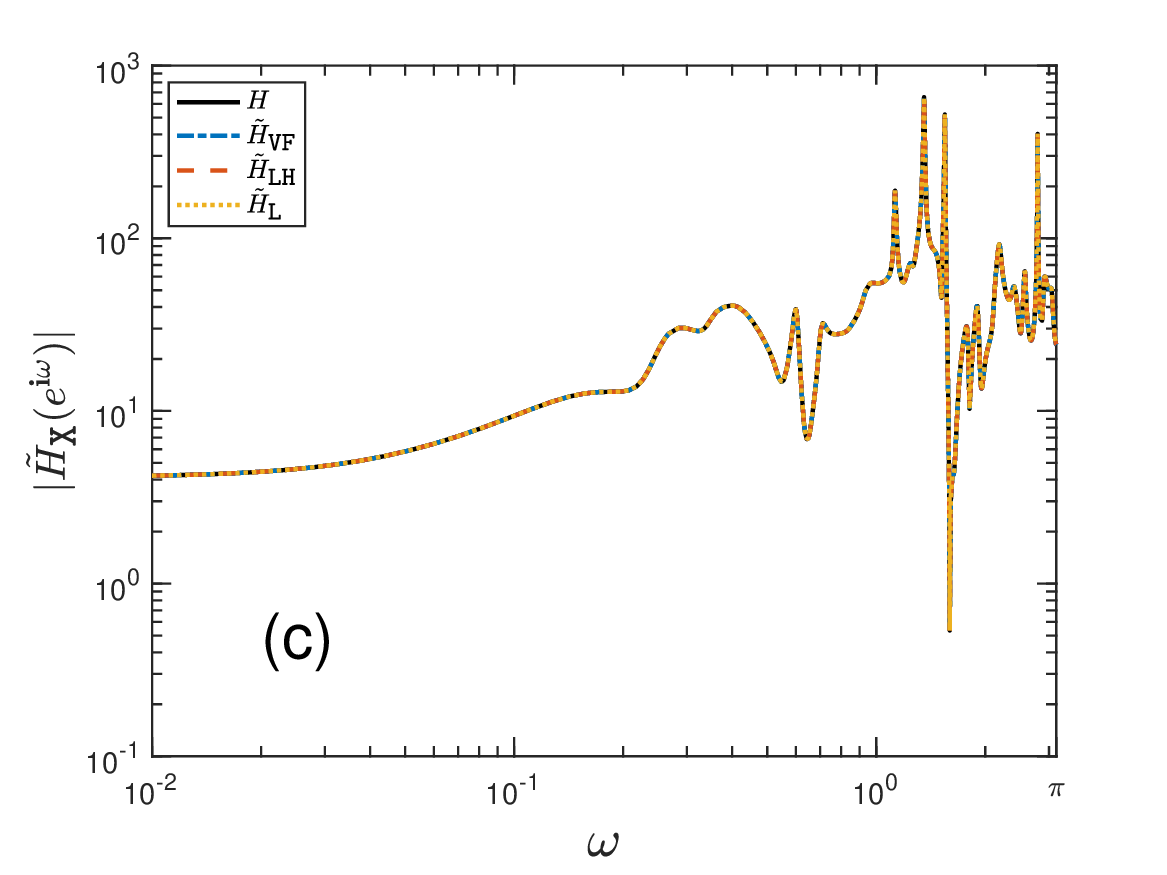}}
     \subfloat{\includegraphics[scale = .32]{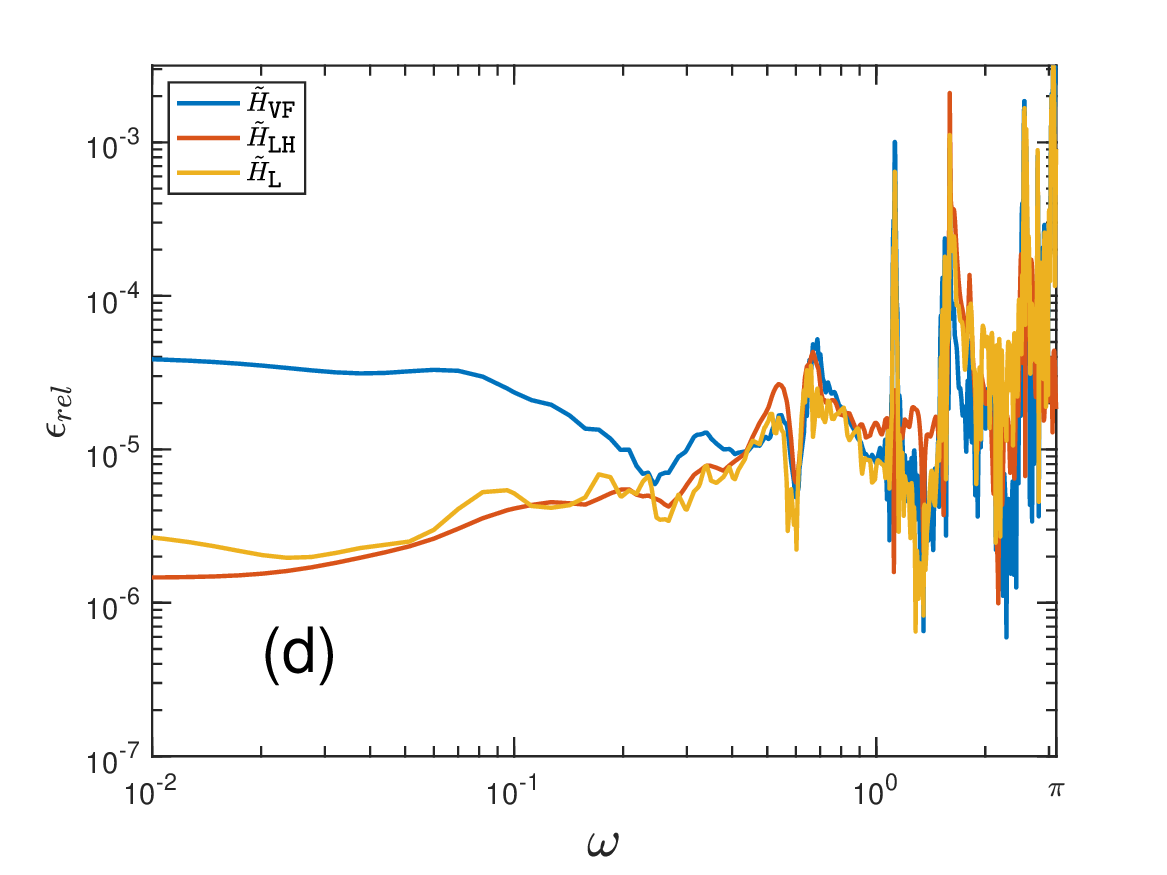}}
     \caption{Frequency responses ((a) and (c)) and point-wise relative errors (plots (b) and (d)) of DDROMs approximating $H_2$, constructed using frequency data recovered from time-domain data via \cref{alg:DataInform} ((a) and (b)) and true frequency data ((c) and (d)).}
     \label{fig:RelErrAndFreqRespSynthROMs}
 \end{figure}

\begin{table}[!htb]
\caption{$\mathcal H_\infty$ errors between the true transfer function $H_2$, DDROMs constructed from true data $\widetilde H_{\texttt X}$, and DDROMs constructed from recovered data $\widehat H_{\texttt X}$}
    \centering
        \begin{tabular}{|c|c|c|c|}
            \hline
             Relative $\mc H_{\infty}$ error & Loewner & Hermite Loewner & Vector Fitting\\
            \hline
            \rule{0pt}{3ex}$\|H_2-\widehat H_{\texttt X}\|_{\mathcal H_{\infty}}/\|H_{2}\|_{\mathcal H_\infty}$&$5.44\times 10^{-4}$&$6.49\times 10^{-5}$ & $4.99\times 10^{-4}$\\[1ex]\hline
            \rule{0pt}{3ex}$\|H_{2}-\widetilde H_{\texttt X}\|_{\mathcal H_{\infty}}/\|H_2\|_{\mathcal H_\infty}$&$5.44\times 10^{-4}$&$6.49\times10^{-5}$ & $4.99 \times 10^{-4}$\\[1ex]\hline
            \rule{0pt}{3ex}$\|\widetilde H_{\texttt X}-\widehat H_{\texttt X}\|_{\mathcal H_{\infty}}/\|\widetilde H_{\texttt X}\|_{\mathcal H_{\infty}}$&$3.16\times 10^{-8}$&$3.18\times 10^{-8}$& $1.48 \times 10^{-8}$\\[1ex]\hline
        \end{tabular}

     \label{tab:H2ErrSynth}
\end{table}

\subsection{Heat Model}
\label{sec:HeatModel}
We now consider the system $\mc S_3,$ a fully discretized model of heat diffusion in a thin rod \cite{Chahlaoui2002Benchmarks}.  This model captures the dynamics of the temperature of a thin rod with heat applied at $1/3$ the length of the rod, with temperature measured at $2/3$ the length of the rod.  The PDE describing the dynamics was discretized to an order $n = 200$ continuous dynamical system, then was discretized in time using the Crank-Nicholson method to obtain $\mathcal S_3$.  For more information, see \cite{Chahlaoui2002Benchmarks}.

We simulate $\mc S_3$ with a Gaussian random input for $T = 1,\!000$ time steps and use \cref{alg:DataInform} with $N = 20$ (calculated by \cite{VerhaegenD1991Subspace1}) to recover estimates $M_0(\sigma_i)$ and $M_1(\sigma_i)$ to $H_3(\sigma_i)$ and $H'_3(\sigma_i)$ where $\sigma_i = e^{\mbf i\omega_i}$, with $\omega_i$ logarithmically spaced in $[10^{-4},\pi)$ and $i = 1,2,\ldots, 500$.  The accuracy of recovered frequency information as calculated by \cref{eq:ErrRecoverdData} is $\varepsilon_0 = 6.44 \times 10^{-9}$ and $\varepsilon_1 = 4.62 \times 10^{-8}$.  

Since $\mc S_3$ is a real system, we again obtain $H(\overline{\sigma_i}) = \overline{H(\sigma_i)}$ at no additional cost.
Using this data, we construct order $r = 10$ approximations using Loewner, Hermite Loewner, and Vector Fitting, $\widehat H_{\texttt{L}}, \widehat H_{\texttt{LH}},$ and $\widehat H_{\texttt{VF}}$, respectively.  We also construct DDROMs from true frequency data ($\widetilde H_{\texttt{L}}, \widetilde H_{\texttt{LH}},$ and $\widetilde H_{\texttt{VF}}$) for comparison.  The relative $\mathcal H_{\infty}$ errors of these models to the true transfer function $H_3$, as well as the relative $\mathcal H_{\infty}$ distances between the DDROMs is displayed in \cref{tab:H2ErrHeat}.  These $\mc H_{\infty}$ errors show that, as in the previuous example, the DDROMs produced from frequency information recovered via \cref{alg:DataInform} are able to mimic the performance of DDROMs built on true frequency response data.

\begin{table}[!htb]
    \centering
    \caption{$\mathcal H_{\infty}$ errors between the true transfer function $H_3$, DDROMs constructed from true data $\widetilde H_{\texttt X}$, and DDROMs constructed from recovered data $\widehat H_{\texttt{X}}$}
        \begin{tabular}{|c|c|c|c|}
            \hline
             Relative $\mc H_{\infty}$ error & Loewner & Hermite Loewner & Vector Fitting\\
            \hline
            \rule{0pt}{3ex}$\|H_3-\widehat H_{\texttt X}\|_{\mathcal H_{\infty}}/\|H_3\|_{\mathcal H_{\infty}}$&$1.67\times 10^{-6}$&$2.50\times 10^{-7}$ & $2.59\times 10^{-7}$\\[1ex]\hline
            \rule{0pt}{3ex}$\|H_3-\widetilde H_{\texttt X}\|_{\mathcal H_{\infty}}/\|H_3\|_{\mathcal H_{\infty}}$&$1.66\times 10^{-6}$&$2.32\times10^{-7}$ & $2.72 \times 10^{-7}$\\[1ex]\hline
            \rule{0pt}{3ex}$\|\widetilde H_{\texttt X}-\widehat H_{\texttt X}\|_{\mathcal H_{\infty}}/\|\widetilde H_{\texttt X}\|_{\mathcal H_{\infty}}$&$3.10\times 10^{-8}$&$2.92\times 10^{-8}$& $6.27 \times 10^{-8}$\\[1ex]\hline
        \end{tabular}
    \label{tab:H2ErrHeat}
\end{table}

\subsection{Penzl's example}
\label{Sec:PenzlEx}
Next, we investigate Penzl's time-continuous linear time invariant system introduced in \cite{Penzl2006Algorithms}.  This benchmark model has been frequently used in the context of reduced order modeling, see, e.g., \cite{Antoulas2005ApproxDynamSys,Ionita2013Thesis, PeherstorferGW2017TDLow}.  
Denote the system as {$\mathcal S^c_4$} and its transfer function as $H^c_4(z)$.
Following \cite{PeherstorferGW2017TDLow}, we discretize {$\mathcal S^c_4$} using the Implicit Euler method with step size $\delta t = 10^{-4}$ to obtain $\mathcal S_4$ with the corresponding transfer function $H_4(z)$.
We simulate $\mc S_4$ with a Gaussian random input for $T = 10,000$ time steps and use \cref{alg:DataInform} to recover estimates $M_0(\sigma_i)$ and $M_1(\sigma_i)$ to $H_4(\sigma_i)$ and $H'_4(\sigma_i)$ where $\sigma_i = e^{\mbf i\omega_i}$, with $\omega_i$ logarithmically spaced in $[-5,\pi)$ and $i = 1,2,\ldots, 140$.

Using \cite{VerhaegenD1991Subspace1} we estimate the system order to be $N = 15$. 
In contrast to the previous two examples, the calculated $N$ does not lead to accurate recovery of frequency information.  The standard deviation error indicator is nearly one for low frequencies, which accurately predicts the large relative errors $\varepsilon_0 = 7.48 \times 10^{-1}$ and $\varepsilon_1 \approx 1$.  {Motivated by the results of \cref{sec:calc_n_hat}, we test several values of $\tilde n > N$, monitoring the corresponding values of $s_W$ for each $\tilde n$.  After increasing the total number of windows to $K = 40$ and iteratively increasing $\tilde n$ while monitoring the error indicator $s_W$, we chose to use $\tilde n = 900$, for which  $s_W$ dropped below $1\%$ for nearly all values of $\omega$.} 

When $\tilde n = 900$ is used (as opposed to $N =15$), the corresponding relative errors were $\varepsilon_0 = 4.48 \times 10^{-3}$ and $\varepsilon_1 = 4.08 \times 10^{-2}$.
This shows the importance of the standard-deviation error indicator and \cref{prop:DontNeedFulln}. Without having any access to true data or true system order, we were able to judge the accuracy of the recovered response for the numerically estimated $N = 15$.  Since the error indicator was large, we used \cref{prop:DontNeedFulln} to try successively higher values of $\tilde n$ in place of $N$ until the estimated error $s_W$ was low enough.  This process resulted in a two order of magnitude reduction of the errors $\varepsilon_0$ and $\varepsilon_1$.

 \begin{figure}
     \centering
     \subfloat{\includegraphics[scale = .32]{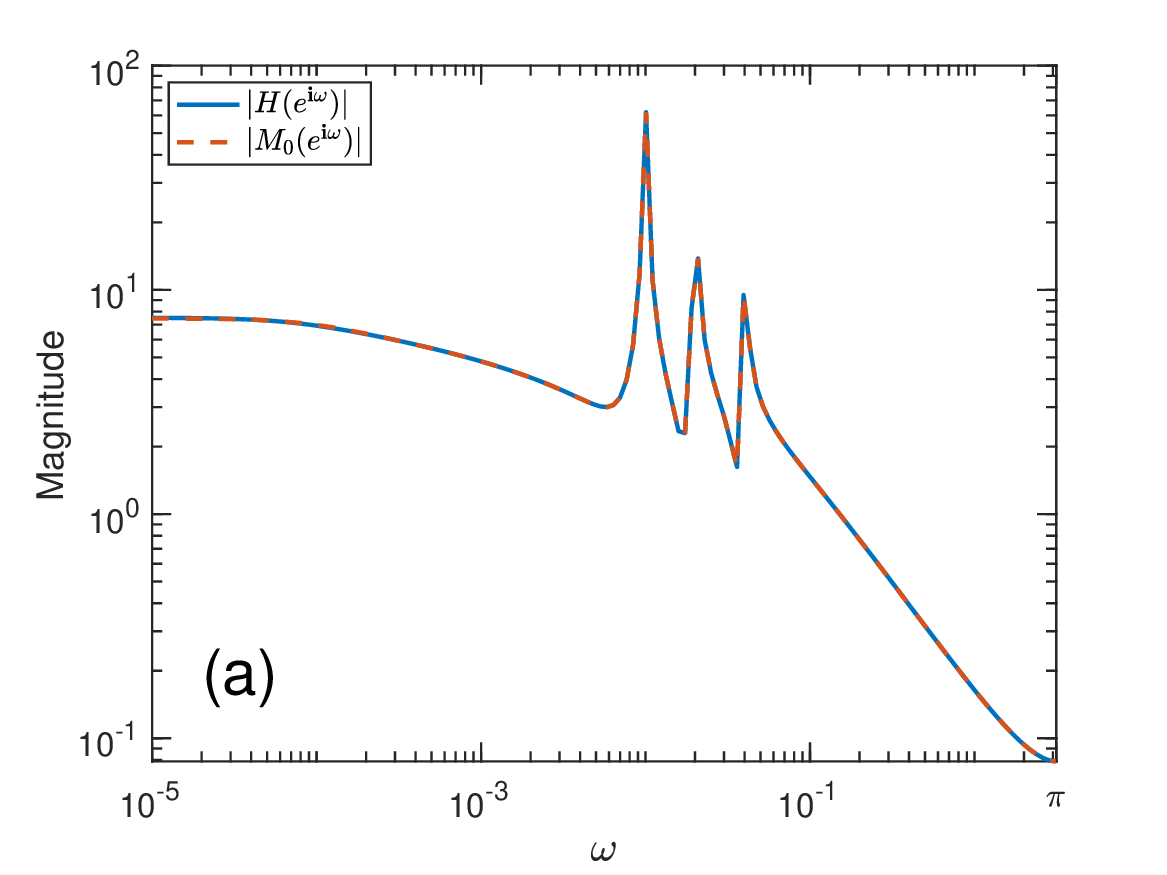}}
     \subfloat{\includegraphics[scale = .32]{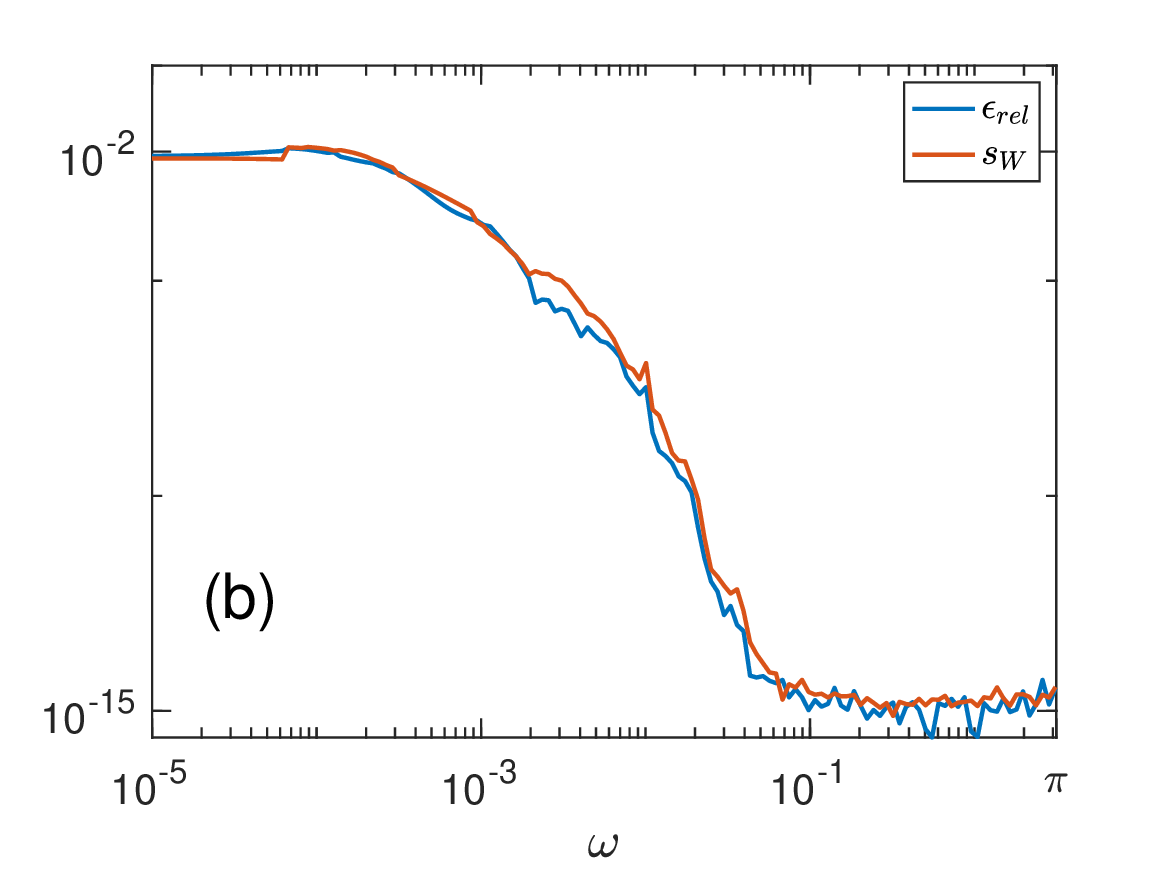}}
     
     \caption{Recovered transfer function values of $H_4$ (a), point-wise relative error of recovered values (b).}
    \label{fig:RelErrAndFreqRespPenzl}
 \end{figure}

We show the accuracy of recovered transfer function values $M_0(e^{\mbf i\omega}) \approx H_4(e^{\mbf i\omega})$ in \cref{fig:RelErrAndFreqRespPenzl}-(a) and the corresponding pointwise relative error in \cref{fig:RelErrAndFreqRespPenzl}-(b).  As these figures show, the data is recovered to the accuracy predicted by $s_W$.  

We now use the recovered values and derivatives of $H_4$ at $\sigma_i$ and $\overline{\sigma_i}$ (obtained at no cost since $H_4(\overline{\sigma_i}) = \overline{H_4(\sigma_i})$) to construct order $r = 14$ DDROMs of $\mc S_4$.  As in the previous sections, we construct DDROMs using Loewner $(\widehat H_{\texttt L})$, Hermite Loewner $(\widehat H_{\texttt{LH}})$, and Vector Fitting $(\widehat H_{\texttt{VF}})$ using the recovered data, as well as their true data counterparts ($\widetilde H_{\texttt L}$, $\widetilde H_{\texttt{LH}}$, and $\widetilde H_{\texttt{VF}}$).  For $\widehat H_{\texttt{VF}}$ only, we use a weighted least-squares problem and set the weight of each recovered transfer function value as $\sqrt[4]{s_W}$.

\begin{figure}
     \centering
     \subfloat{\includegraphics[scale = .32]{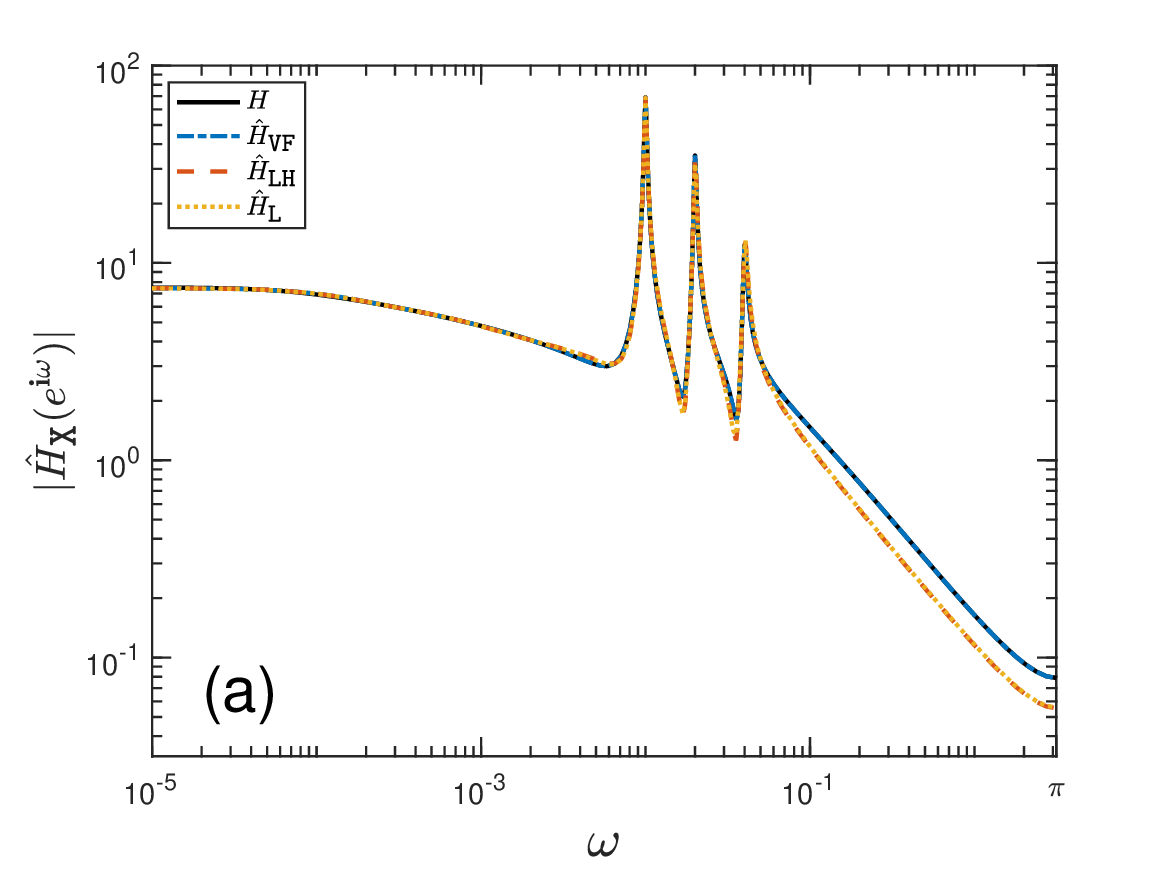}}
     \subfloat{\includegraphics[scale = .32]{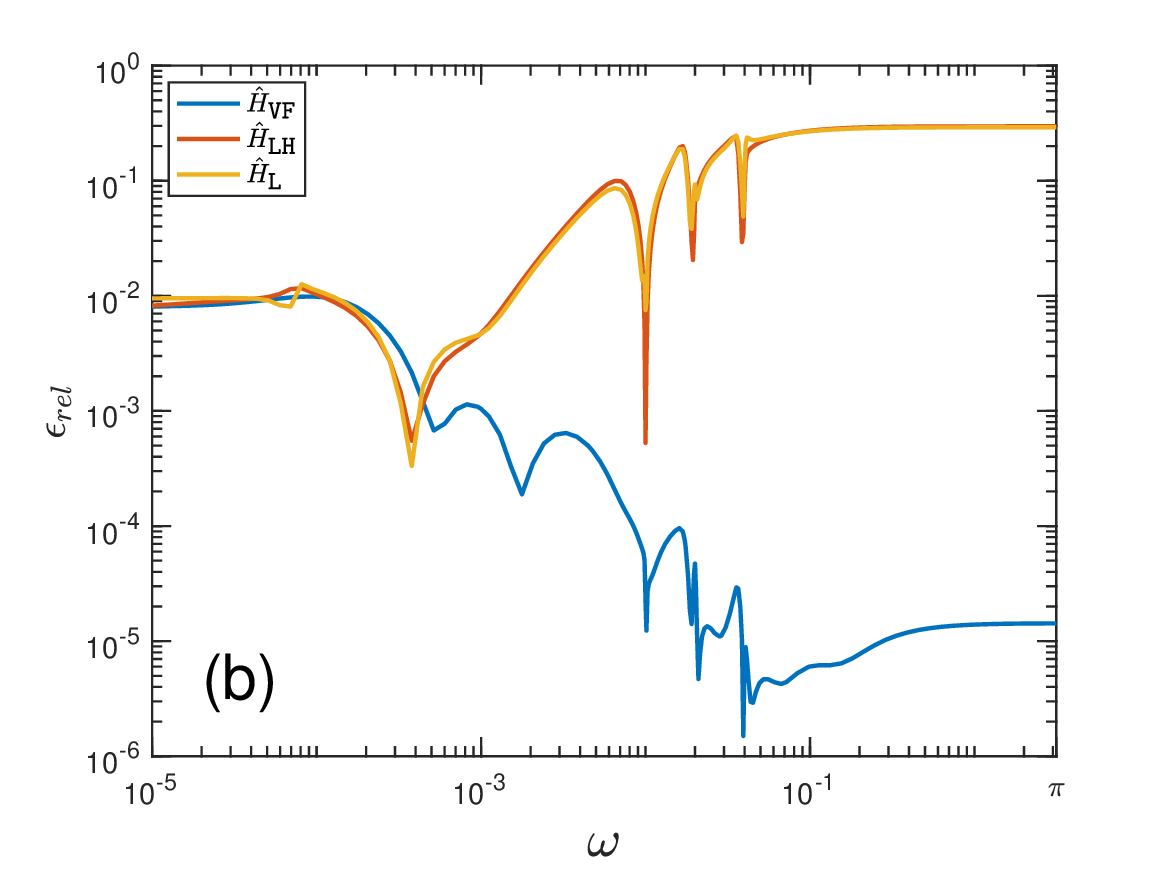}}
     
     \subfloat{\includegraphics[scale = .32]{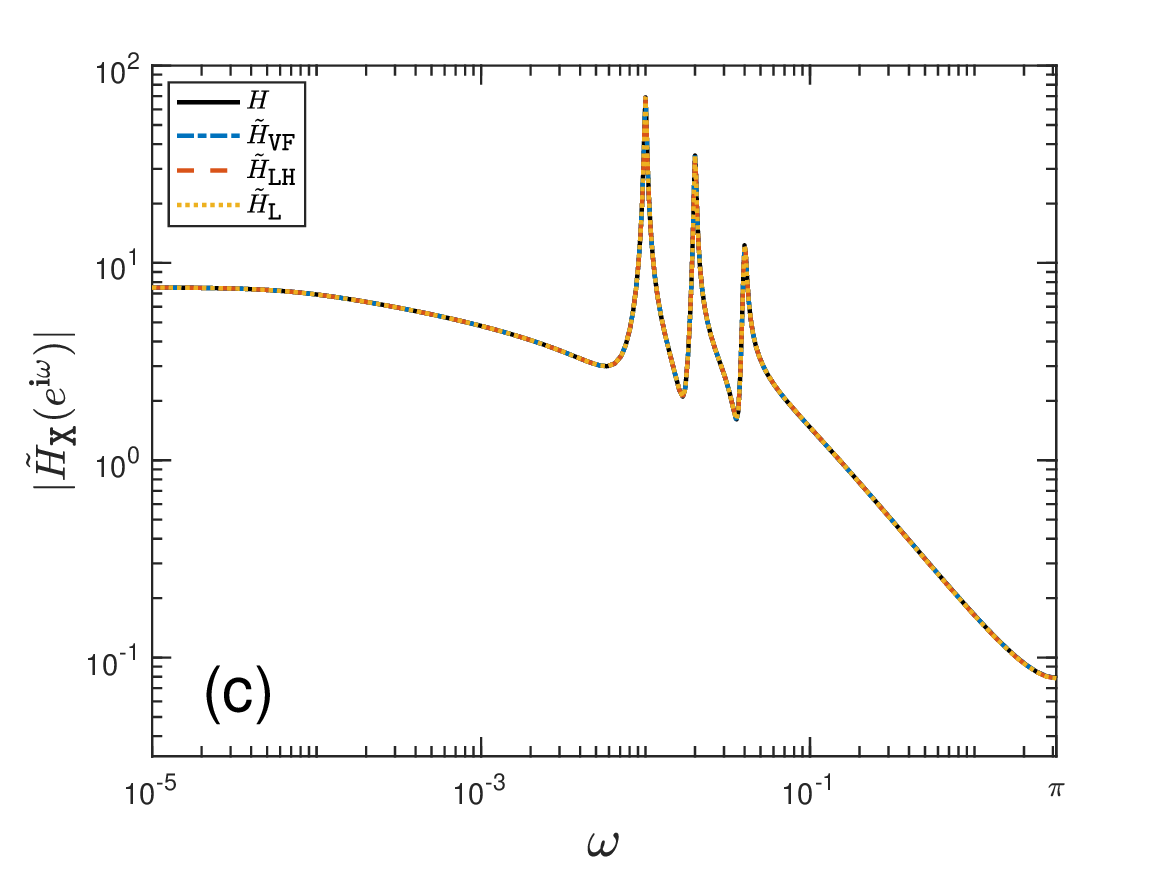}}
     \subfloat{\includegraphics[scale = .32]{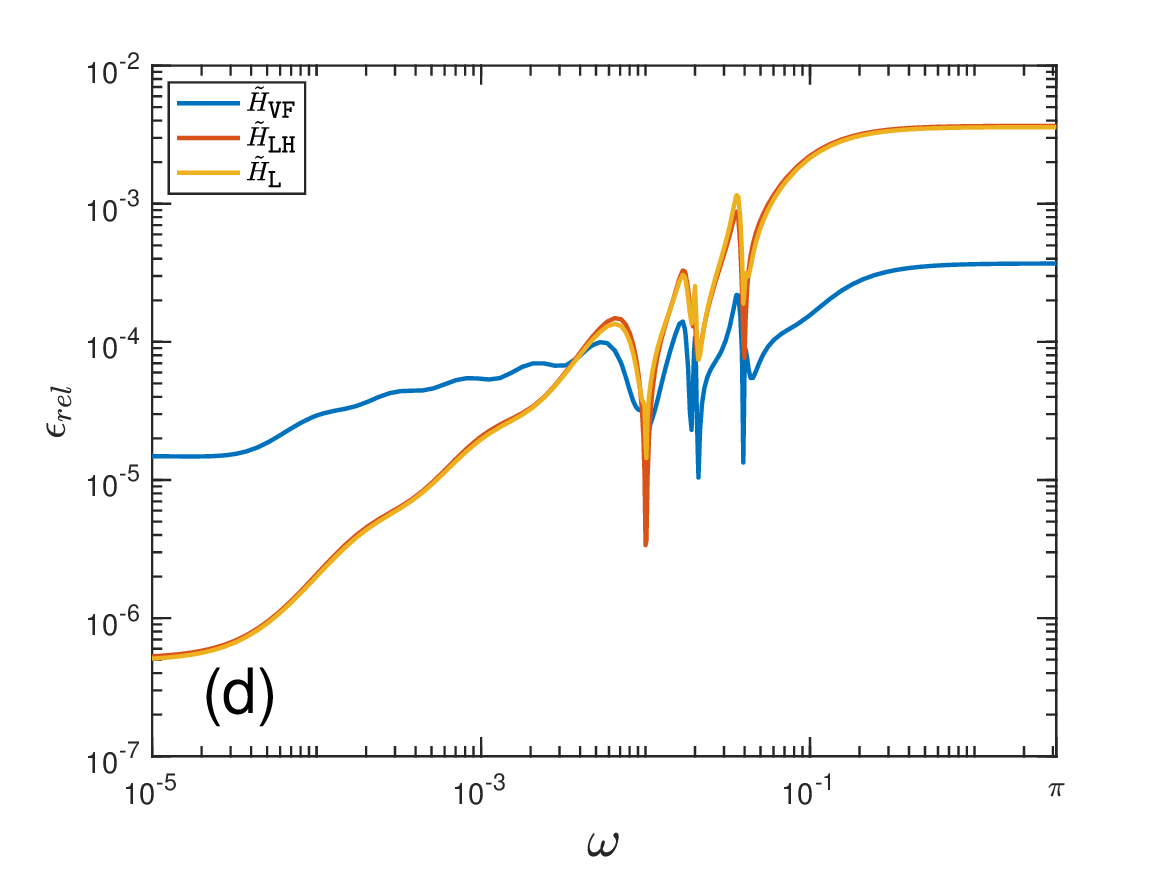}}
     \caption{Frequency responses ((a) and (c)) and point-wise relative errors ((b) and (d)) of DDROMs approximating $H_4$,  using frequency data recovered from time-domain data via \cref{alg:DataInform} ((a) and (b)) and true frequency data ((c) and (d)).}
     \label{fig:RelErrAndFreqRespPenzlROMs}
 \end{figure}

The amplitude frequency response of $H_4$ and the three DDROMs constructed from recovered frequency information, i.e., $\widehat H_{\texttt L}, \widehat H_{\texttt{LH}},$ and  $\widehat H_{\texttt{VF}}$, are shown in \cref{fig:RelErrAndFreqRespPenzlROMs}-(a).  For this example, we see that unlike the previous cases, these DDROMs did not perform  as well as the DDROMs constructed from true data (shown in \cref{fig:RelErrAndFreqRespPenzlROMs}-(c)).  This is also reflected in the error plots, \cref{fig:RelErrAndFreqRespPenzlROMs}-(b) for DDROMs from recovered data and \cref{fig:RelErrAndFreqRespPenzlROMs}-(d) for DDROMs from true data.  Quantitatively, the $\mathcal H_{\infty}$ errors shown in \cref{tab:H2ErrPenzl} show that 
even though DDROMs learned from recovered frequency data provide reasonably accurate approximations to the original system (especially true for $\widehat H_{\texttt{VF}}$ with a relative accuracy of $10^{-3}$), they have degraded accuracy compared to the corresponding DDROM learned from true data.  The Vector Fitting model was affected the least (see the last row of \cref{tab:H2ErrPenzl}).
These observations are not  surprising.  While the Loewner and Hermite Loewner models are interpolatory, Vector Fitting model is least-squares based, thus is expected to be more robust to the perturbations/inaccuracies in the data (as is the case here).

For the Vector Fitting model from recovered data ($\widehat H_{\texttt{VF}}$) the degraded accuracy can be explained by the lower accuracy of the recovered frequency information at small frequencies (see \cref{fig:RelErrAndFreqRespPenzl}-(b)).  Indeed, \cref{fig:RelErrAndFreqRespPenzlROMs}-(b) shows that $\widehat H_{\texttt{VF}}$ has the lowest accuracy at lower frequencies, and actually outperforms $\widetilde H_{\texttt{VF}}$ at high frequencies.  This increase in performance at high frequencies is expected to be due to the weighting used to construct $\widehat H_{\texttt{VF}}$. 

Both $\widehat H_{\texttt L}$ and $\widehat H_{\texttt{LH}}$, the Loewner and Hermite Loewner model formed from recovered frequency data were unstable.  The frequency response and errors for these DDROMs shown in \cref{fig:RelErrAndFreqRespPenzlROMs} as well as the $\mc H_{\infty}$ errors reported in \cref{tab:H2ErrPenzl} are calculated from the stable part of these systems, which were order 12 and order 11, respectively.  In addition, we note that the amplitude frequency response of these systems significantly differs from that of $H_4$ (and all other DDROMs) for high frequencies.  This cannot be explained by the accuracy of the recovered frequency information since the recovered frequency information is highly accurate in the high frequency regime.  

The recent work~\cite{DramacP2020NoisyLoewner} that studied the Loewner-based modeling in the case of perturbed/noisy data could provide some hints in better understanding this behavior. Moreover, the work~\cite{BeattieGWyatt2012Inexact} connected the impact of inexact solves in projection-based interpolatory model reduction to a backward error analysis in Loewner based modeling. This could also shed some light on the behavior of Loewner models.
Investigating these findings in more detail is left to a future work where a more detailed benchmark comparison of many frequency-based modeling techniques will be present. In this paper, our main focus was on the data-informativity framework itself.

\begin{table}[!htb]
    \centering
    \caption{$\mathcal H_{\infty}$ errors between the true transfer function $H_4$, DDROMs constructed from true data $\widetilde H_r$, and DDROMs constructed from learned data $\widehat H_{\texttt X}$}
        \begin{tabular}{|c|c|c|c|}
            \hline
             Relative ${\mathcal H_{\infty}}$ error & Loewner & Hermite Loewner & Vector Fitting\\
            \hline
            \rule{0pt}{3ex}${\|H_4-\widehat H_r\|_{\mathcal H_{\infty}}}/{\|H_4\|_{\mathcal H_{\infty}}}$&$4.71\times 10^{-2}$&$3.80\times 10^{-2}$ & $1.01\times 10^{-3}$\\[1ex]\hline
            \rule{0pt}{3ex}$\|H_4-\widetilde H_r\|_{\mathcal H_{\infty}}/\|H_4\|_{\mathcal H_{\infty}}$&$1.29\times 10^{-4}$&$7.33\times10^{-5}$ & $5.43 \times 10^{-5}$\\[1ex]\hline
            \rule{0pt}{3ex}$\|\widetilde H_r-\widehat H_r\|_{\mathcal H_{\infty}}/\|\widetilde H_r\|_{\mathcal H_{\infty}}$&$4.72\times 10^{-2}$&$3.81\times 10^{-2}$ & $1.01\times 10^{-3}$\\[1ex]\hline
        \end{tabular}
    \label{tab:H2ErrPenzl}
\end{table}
\section{Conclusions}
{We presented an analysis of the data informativity approach to moment matching. Our analysis shows how to construct theoretically equivalent but  better conditioned linear systems used in obtaining frequency data from time-domain data, develops an error indicator, and removes the assumption that the system order is known.  This analysis leads to a robust algorithm for recovering frequency information from a single time-domain trajectory of a linear dynamical system.  These recovered frequency data were then used to construct DDROMs via the Loewner and Hermite Loewner frameworks and Vector Fitting. The least-square formulation via Vector Fitting  showed little sensitivity to the errors in the recovered frequency information.}

Future work includes using this recovered frequency information to construct $\mathcal H_2$ optimal reduced order models via the Iterative Rational Krylov Algorithm. In this work,  we used simulated time-domain data. We wish to explore the effectiveness of our algorithm on data that come from physical experiments. 

\bibliographystyle{siamplain}
\bibliography{references2}

\end{document}